\documentclass[12pt, reqno]{amsart}

\usepackage{latexsym,amsfonts,amsmath,amssymb,amsthm,url,graphics,psfrag,amscd,stmaryrd, mathrsfs}
\usepackage[english]{babel}
\usepackage[T1]{fontenc}
\usepackage{mathrsfs}
\usepackage{graphics}
\usepackage{graphicx}
\usepackage{color}

\numberwithin{equation}{section}

\newcommand{\R}{\mathbb{R}}

\newcommand{\E}{\mathbb{E}}

\newcommand{\kD}{\mathcal{D}}

\newcommand{\kO}{\mathcal{O}}

\newcommand{\kQ}{\mathcal{Q}}

\newcommand{\kM}{\mathcal{M}}

\newcommand{\kH}{\mathcal{H}}

\newcommand{\kN}{\mathcal{N}}
\newcommand{\kX}{\mathcal{X}}
\newcommand{\kL}{\mathcal{L}}
\newcommand{\kU}{\mathcal{U}}

\newcommand{\en}{\mathbb{E} }

\newcommand{\atanh}{\textrm{atanh}}

\newcommand{\lin}{\left[\kern-0.15em\left[}
\newcommand{\rin} {\right]\kern-0.15em\right]}
\newcommand{\linf}{[\kern-0.15em [}
\newcommand{\rinf} {]\kern-0.15em ]}
\newcommand{\ilin}{\left]\kern-0.15em\left]}
\newcommand{\irin} {\right[\kern-0.15em\right[}

\newtheorem{lem}{Lemma}[section]

\newtheorem{prop}[lem]{Proposition}
\newtheorem{theo}[lem]{Theorem}

\newtheorem {rem}[lem] {Remark}

\newtheorem*{ack}{Acknowledgments}

\title[ Ising model on random regular graphs]
       {\bf Annealed  limit theorems for the Ising model on random regular graphs}

\author{Van Hao Can}
\address{Aix Marseille Universit\'e, CNRS, Centrale Marseille, I2M, UMR 7373, 13453 Marseille, France}
\address{Institute of Mathematics, Vietnam Academy of Science and Technology, 18 Hoang Quoc Viet, 10307 Ha Noi, Viet Nam}
\email{cvhao89@gmail.com}


\begin{document}
\maketitle
\begin{abstract}
In a recent paper \cite{GGHPb},  Giardin\`a, Giberti, Hofstad, Prioriello have proved a law of large number and a central limit theorem with respect to the annealed measure for the magnetization of the Ising model on some random graphs including the random 2-regular graph. We  present a new proof of their  results, which applies to all random regular graphs. In addition,  we prove the existence of  annealed pressure in the case of configuration model random graphs.    
\end{abstract}
\section{Introduction}
The ferromagnetic Ising model is one of the most well-known models in statistical physics describing  cooperative behaviors. In this model,  each vertex in a graph is assigned by one spin that can be one of two states +1, or -1, while the configuration probability is given by the Boltzmann-Gibbs measure. These spins cooperatively interact with each other toward alignment: spins of vertices connected by edges tend to be at the same state. 

\vspace{0.2cm}
The Ising model on   regular lattices has been studied carefully by many authors, resulting in numerous beautiful results, see e.g. \cite{G,L}.  Recently, a lot of attention has been draw into investigating this model on class of random graphs \cite{AB, BD, C,  DM,DMSS, D,  DGH, DGGHP, DHJN, DGM, GGHPa, GGHPb, MMS, MS}. In the new framework, the source of randomness is the combination of the  law of spin configurations and the law of  random graphs. Beside of generalizing class of graphs, some authors try to consider different types of   configuration probability.  Most of previous studies focused on the {\it quenched} setting, in which  a graph sample is fixed then the configuration probability is defined according to this realization of the graph. In a recent paper \cite{GGHPb},   Giardin\`a et al.  consider an {\it annealed} setting, where  the configuration probability is defined by taking into account the information of all  graph samples.  More precisely, they define the annealed Ising model  as follows.

\vspace{0.2 cm}
 Let $G_n$ be a random multi-graph, that is a random graph possibly having self-loops and multiple edges between two vertices,  with  $n$ vertices  $v_1, \ldots,v_n$.  Let $\Omega_n=\{-1,1\}^n$ be the space of spin configurations. For any  $\sigma=(\sigma_1, \ldots, \sigma_n) \in \Omega_n$, its energy  is given by the Hamiltonian function: 
\[H(\sigma)= -\beta  \sum_{ i \leq j} k_{i,j} \sigma_i \sigma_j - B \sum_{i=1}^n \sigma_i,\]
where $k_{i,j}$ is the number of edges between $v_i$ and $v_j$, where $\beta \geq 0$ is the inverse temperature and  $B \in \R$ is the uniform external magnetic field.  

Then the configuration probability is given by what they call the {\it annealed measure}:
\[\mu_n(\sigma)= \frac{\en(\exp (-H(\sigma)))}{\E(Z_n(\beta,B))},\] 
 where $\E$ denotes the expectation with respect to the random graph, and $Z_n(\beta,B)$ is the partition function:
\[Z_n(\beta,B)= \sum\limits_{\sigma \in \Omega_n} \exp( -H(\sigma)).\]

%As mentioned in \cite{GGPH}, the complexity of the annealed measure  comes from the dependency of the random variables $(k_{i,j})$.

 In \cite{GGHPb}, the authors study this  Ising model  on the rank-one inhomogeneous random graph, the random 2-regular graph and the configuration model with degrees $1$ and $2$.  After determining limits of thermodynamic quantities and the critical inverse temperature, they prove  laws of large numbers (LLN) and central limit theorems (CLT) with respect to the  annealed measure for the total spin $S_n= \sigma_1 + \ldots + \sigma_n$.  Our main contribution in this paper is to   generalize their results  to the class of all random regular graphs and prove the existence of annealed pressure in the case of the  configuration model -  a generalization of  the random regular graph, see Section 2.1 for a definition.
% . In fact, it is   a  well known model of random graph with any prescribed degree distribution, 
\subsection{Main results}First, we  give some  definitions  following \cite{GGHPb} of   the thermodynamic quantities in finite volume.
\begin{itemize}
\item[(i)] The annealed  pressure is given by
\[\psi_n(\beta,B)=\frac{1}{n} \log \E(Z_n(\beta,B)).\]
\item[(ii)] The annealed magnetization is given by
\[M_n(\beta,B)=\E_{\mu_n} \left( \frac{S_n}{n}\right),\]
where $S_n = \sigma_1 + \ldots+\sigma_n$. After a simple computation, we get 
\[M_n(\beta,B)= \frac{\partial}{\partial B} \psi_n(\beta,B).\]
\item[(iii)] The annealed susceptibility is given by
\[\chi_n(\beta,B)=\textrm{Var}_{\mu_n} \left( \frac{S_n}{\sqrt{n}}\right).\]
We also can  prove that
\[\chi_n(\beta,B)= \frac{\partial}{\partial B} M_n(\beta,B) =\frac{\partial^2}{\partial B^2} \psi_n(\beta,B).\]
\end{itemize} 
When the sequence $(M_n(\beta,B))_n$ converges to a limit, say $\kM(\beta,B)$, we define the spontaneous magnetization as $\kM(\beta,0^+)= \lim \limits_{B \searrow 0} \kM(\beta,B)$. Then the critical inverse temperature is defined as 
\[\beta_c = \inf \{\beta>0: \kM(\beta,0^+)>0 \}.\] 
 Finally,  the region of the existence of the limit magnetization is defined as
\[\kU= \{(\beta,B): \beta \geq 0, B \neq 0 \textrm{ or } 0<\beta < \beta_c, B=0\}.\]
 Now, we may  introduce our results in the case of random regular graphs. First, we show  the  limits of  thermodynamic quantities when the number of vertices tends to infinity.
\begin{theo} \label{ttd}(The thermodynamic limits). Let us consider the Ising model on the random $d$-regular graph with $d\geq 2$. Then the following assertions hold.
\begin{itemize}
\item[(i)] For all $\beta \geq 0$ and $B \in \R$, the annealed pressure converges 
\begin{eqnarray*}
 \lim_{n\rightarrow \infty} \psi_n(\beta,B) &=&\psi(\beta,B) \\
  &=& \frac{\beta d}{2}  -B+ \max_{0 \leq t \leq 1} \left[ (t-1) \log(1-t)- t \log t + 2Bt + d F(t) \right], 
\end{eqnarray*}
where 
\begin{equation*} \label{fut}
F(t)= \int_0^{u(t)} \log f(s)ds,
\end{equation*}
with $u(t)= \min \{t,  1-t\}$  and
\begin{eqnarray*}
f( s) = \frac{e^{-2\beta}(1-2s)+ \sqrt{1+(e^{-4\beta}-1)(1-2s)^2}}{2(1-s)}.
\end{eqnarray*}
\item[(ii)] For all $(\beta, B) \in \kU$, the magnetization  converges 
\begin{equation*}
\lim_{n\rightarrow \infty}M_n(\beta,B) =\kM(\beta,B)= \frac{\partial}{\partial B} \psi(\beta,B). 
\end{equation*} 
Moreover, the critical inverse temperature is 
\begin{displaymath}
\beta_c= \emph{\atanh}(1/(d-1))=\left \{ \begin{array}{ll}
\frac{1}{2} \log \left( \frac{d}{d-2} \right) & \textrm{ if } \quad  d \geq 3  \\
\infty  & \textrm{ if } \quad  d= 2.  
\end{array} \right.
\end{displaymath}
\item[(iii)] For all $(\beta,B) \in \kU$,  the annealed susceptibility converges 
\[ \lim_{n\rightarrow \infty}\chi_n(\beta,B) = \chi(\beta,B)= \frac{\partial^2}{\partial B^2} \psi(\beta,B) .\]
\end{itemize} 
\end{theo}
Based on the thermodynamic limits theorem, we obtain  a  law of large number and a central limit theorem for the total spin. 
\begin{theo} \label{tll}(Annealed  LLN). 
Suppose that $(\beta,B) \in \kU$. Then for any $\varepsilon >0$, there exists a positive constant $L=L(\varepsilon)$, such that for all sufficiently large $n$
\[\mu_n \left( \, \vline \frac{S_n}{n} - \kM(\beta,B) \vline \, > \varepsilon \right) \leq \exp(-nL),\]
where $\kM(\beta,B)$ is defined in Theorem \ref{ttd} (ii). 
\end{theo}
\begin{theo}  \label{tclt} (Annealed CLT). For all $(\beta,B) \in \kU$, the total spin under the annealed measure satisfies a central limit theorem:
\[\frac{S_n - \E_{\mu_n}(S_n)}{\sqrt{n}} \quad \mathop{\longrightarrow}^{(\kD)} \quad \kN(0, \chi(\beta,B)) \qquad \textrm{ w.r.t. } \mu_n,\]
where $\chi(\beta,B)$ is defined in Theorem \ref{ttd} (iii) and  $\kN(0, \chi)$ denotes a centered Gaussian random variable with variance $\chi$.  
\end{theo}
In the low temperature regime and in the absence of external field, the magnetization does not converges to a constant. However, similar to Curie-Weiss model, the law of magnetization converges to a combination of two Dirac's measures.  
\begin{prop} \label{pnl} 
Suppose that $\beta> \beta_c$ and $B=0$. Then there exists a positive constant $\nu=\nu(\beta)$, such that as $n \rightarrow \infty$,
\[\mu_n \left( \, \vline \, \frac{S_n}{n} - \nu \,  \vline \, \leq n^{-1/6}  \right)  \longrightarrow 1/2 \hspace{0.9 cm} \textrm{and } \hspace{0.9 cm} \mu_n \left( \, \vline \, \frac{S_n}{n} + \nu \,  \vline \, \leq n^{-1/6}  \right)  \longrightarrow 1/2. \]
\end{prop}
Our result on the existence of annealed pressure in the case of the configuration model with general degree distributions is stated  in Section 7, due to its complexity.
\subsection{Discussion}   One  challenge   in the annealed setting is that we have to take into account  all  graph samples. There are probably some rare  samples that  give a non-trivial contribution. Studying them   often links to a very challenging topic, the  large deviation properties  of  random graphs. Let us give here some comments on the approach, consequence and extension of our results. 

\vspace{0.2 cm}
\noindent (i) On the strategy of proofs. Structure of the random $d$-regular graph  strongly depends on $d$. When $d$ increases, the graph becomes more and more complicated.  In the case  $d=2$, the annealed Ising model on the graph is well studied  in  \cite{GGHPb}. Their approach is based on the fact  that every random $2$-regular graph   consists of a collection of cycles and the partition function on a cycle can be computed explicitly. 
However, when $d \geq 3$,  this particular fact does not hold anymore. On the other hand, we  realize  that for any spin configuration, its Hamiltonian can be expressed in terms of $\beta$, $B$ and the number of disagreeing edges (the edges whose two  extremities have different spins). Moreover, by the symmetry  in term of law  of random regular graphs, for any pair of configurations  with the same number of positive spins, these numbers of disagreeing edges have the same distribution. Thus the Halmitonians of these configurations have the same law. Hence  we show that the expectation of the partition function has the form  $\sum _{i \leq n} \binom{n}{i} \theta(i, \beta, B)$. Furthermore,
\begin{equation*}
\frac{1}{n}\log \sum _{i=0}^n \binom{n}{i} \theta(i, \beta, B) =  \max_{0 \leq i \leq n} \frac{1}{n} \log \left[ \binom{n}{i} \theta(i, \beta, B)\right] + o(1).
\end{equation*}
 This explains the form of the annealed pressure $\psi(\beta,B)$  in Theorem  \ref{ttd} (i), which  somehow looks like  a large deviation result.

\vspace{0.2 cm}
 To prove the limit theorems, we use the same general strategy as in \cite{GGHPa, GGHPb}.  More precisely, we define the sequence of  cumulant generating functions as
\[c_n(t)= \frac{1}{n} \log \E_{\mu_n} (\exp(tS_n))= \psi_n(\beta,B+t) - \psi_n(\beta,B).\]
Then by Theorem \ref{ttd}, this sequence converges to 
\[ c(t)= \psi(\beta, B+t) - \psi(\beta, B).  \]
In \cite[Sections 2.1 and 2.2]{GGHPa}, the authors show  %(see also \cite[Theorem II.6.3]{E} and \cite[Section 2.2]{GGHPb})
  that if the function  $c(t)$ is differentiable at $0$ then the sequence $(S_n/n)_n$ converges in probability exponentially fast to $c'(0)$ w.r.t $\mu_n$.  That means,  for any real number $\varepsilon >0$, there exists a positive constant $L=L(\varepsilon)$, such that  for all $n$ large enough
\[\mu_n \left( \, \vline \frac{S_n}{n} - c'(0) \vline \, > \varepsilon \right) \leq \exp(-nL).\]
We will show in Section 4 that the function $\psi(\beta,B)$ is differentiable with respect to  $B$. Thus $c(t)$ is differentiable and the annealed LLN follows. 

On the other hand,  by using Theorem A.8.7 (a) in \cite{E},  the central limit theorem in Theorem \ref{tclt} follows from the convergence of generating function of the normalized sum, i.e. for any fixed number $t>0$, 
\begin{equation*} \label{dd}
\E_{\mu_n} \left( \exp \left( \frac{t(S_n - \E_{\mu_n} (S_n))}{\sqrt{n}} \right) \right)\longrightarrow \exp \left( \frac{\chi(\beta,B) t^2}{2} \right).
\end{equation*}
The authors in \cite[Section 3.2]{GGHPb} show that this convergence holds  if  the following condition is  satisfied:  For any fixed number $t>0$ and for any sequence $(t_n)$ satisfying $t_n \in [0, t/ \sqrt{n}]$, one has
\begin{equation*} 
c_n''(t_n) \rightarrow \chi(\beta, B).
\end{equation*}
 We refer to  Lemma \ref{lcgf} for the proof of this condition.

\vspace{0.2 cm}
\noindent (ii) On the case $d=2$. We show in Proposition \ref{eopp} that with $d=2$, the annealed pressure is exactly solved and agrees the result obtained in \cite{GGHPb}, where the limit theorems have been proved. Hence, in Sections 4, 5, 6 we only study limit theorems for the case  $d\geq 3$. 

\vspace{0.2 cm}

\noindent (iii) On the similarity to the quenched Ising model. Theorem \ref{ttd} (ii) shows that the annealed Ising model undergoes a phase transition  at the critical inverse temperature $\beta_c = \textrm{atanh}(1/(d-1))$, which is equal to the critical value  of the quenched Ising model. Moreover,  we will prove in Proposition \ref{eopp}   that the annealed  and quenched pressures are actually the same.  As a consequence, all the thermodynamic limits of the  annealed and quenched  models are identical, and these two models should behave  alike. In fact,  limit theorems similar to our results have been proved for quenched model in \cite{GGHPa}.  This similarity has been conjectured in  \cite[Section 1.5.1]{GGHPb}.  

\vspace{0.2 cm}
\noindent (iv) On the generalizations. In Section 6, we study the Ising model on   the configuration model with general degree distributions.  Comparing with  the case of  random regular graphs, we have additionally a source of randomness coming from the sequence of degrees. This randomness makes the problem  much more difficult.  In particular, the annealed pressure  obtained in Proposition \ref{pbbc} is  so complicated that we can not even prove its differentiability. Without the differentiability, we can not go further to the other thermodynamic quantities or limit theorems. 

Another natural question is to generalize our result to the Potts model where the spin of   vertex may take  $q$ values with $q\geq 3$, and the Hamiltonian is proportional to the number of agreeing edges.  Our method  possibly applies  for this  model, but it would require much  work.   The symmetry property  that the measures of configurations with similar structure of spins   are equal   will continue to hold for the Potts model. However, the Hamiltonian of configurations is more complicated than that of Ising model. Indeed, there are now $q(q-1)/2$ types of disagreeing edges instead of $1$ type as in Ising model. Hence a recursive relation between agreeing and disagreeing edges would be much harder than  the one for Ising model obtained  in Section 2.

\vspace{0.2 cm}
\noindent (v) On the organization of the  paper.  In section 2, we give a definition of the configuration model and prove a key lemma for random 1-regular graph  used in the proof of the existence of annealed pressure. In section 3, we study the annealed pressure and prove Theorem \ref{ttd} (i). In Section 4,  we consider the magnetization,  prove Theorem \ref{ttd} (ii) and Theorem \ref{tll}. In Section 5, we prove Theorem \ref{ttd} (iii) and Theorem \ref{tclt}. In Section 6, we prove Proposition \ref{pnl}.  In Section 7,  we prove the existence of the annealed pressure in the case of general configuration models.  Appendix is devoted to prove some technical points of our proofs. 
\section{Preliminaries}
\subsection{Configuration model}  Let us give  a definition following \cite{H} of the  configuration model with prescribed degree sequence.  For each $n$, let  $V_n=\{v_1, \ldots, v_n\}$ be the vertex set of a graph $G_n$, let ${\bf D}= (D_i)_{i \leq n}$ be a sequence of integers.  We  construct the edge set of $G_n$ as follows.	 First, we assume that $\ell_n =\sum_1^n D_i$ is even (if not increase one of the $D_i$'s by $1$,  which makes no difference in what follows). For each vertex $v_i$, start with $D_i$ half-edges incident to $v_i$.  Then we denote by $\kH$ the set of all the half-edges. Select one of them $h_1$ arbitrarily and then choose a half-edge $h_2$ uniformly from $\kH \setminus \{h_1\}$, and match $h_1$ and $h_2$ to form an edge. Next, select arbitrarily another half-edge $h_3$ from $\kH \setminus \{h_1, h_2\}$  and match it to another $h_4$ uniformly chosen from  $\kH \setminus \{h_1, h_2, h_3\}$. Then continue this procedure until there are no more half-edges.  We finally get a multiple random graph that may have self-loops and multiple edges between vertices satisfying the degree of $v_i$ is $D_i$ for all $i$. We denote the obtained graph by $G_n({\bf D})$.

\vspace{0.2 cm}
For $d\geq 1$, if $D_i =d$ for all $i =1, \ldots, n$  we call $G_n({\bf D})$ the {\it random $d$-regular graph}, and denote it by $G_{n,d}$. The random 1-regular graph will be employed several times in the proofs, so we distinguish  its set of vertices  with that of the $G_{n,d}$.  More precisely, we denote by $\bar{V}_m=\{w_1,\ldots,w_m\}$ the set of vertices of $G_{n,1}$. 

\vspace{0.2 cm}
We now explain the role of $G_{n,1}$ in our arguments. We  show in \eqref{hnsg} that the Hamiltonian of a given configuration can be expressed in term of the number of disagreeing edges. By the construction of the configuration model, we have a relation between the number of disagreeing edges of $G_n({\bf D})$ and that of  $G_{\ell_n,1}$ with $\ell_n = D_1 + \ldots + D_n$. More concretely, for $A\subset V_n$, let us denote  
\[e(A,A^c)= \#\{\textrm{edges between $A$ and $A^c$ in } G_n({\bf D}) \}.\]
On the other hand, for each integer $m$, let $\bar{V}_{m} = \{w_1, \ldots, w_{m}\}$ be the vertex set of $G_{m,1}$. For any $1 \leq k \leq m$, we define $\bar{U}_{k}= \{w_1, \ldots, w_{k}\}$,  $\bar{U}_{k}^c=\bar{V}_{m} \setminus \bar{U}_{k}$ and 
\begin{equation} \label{dfox}
X(k,m) = \#\{\textrm{edges between $\bar{U}_{k}$ and $\bar{U}_{k}^c$ in } G_{m,1} \}.
\end{equation}  
It directly follows from  the construction of the configuration model that   
 \begin{equation} \label{hebn}
e(A, A^c) \mathop{=}^{(\kD)} X(\ell_A, \ell_n),  
\end{equation}
where 
\[\ell_A= \sum \limits_{i= 1} ^ n D_i 1(v_i \in A) \hspace{2 cm} \textrm{and} \hspace{2cm} \ell_n = \ell_{V_n}.\]
The relation \eqref{hebn} allows us to reduce problems on disagreeing edges   of configuration models (or Hamiltonian of Ising model) to the one of random 1-regular graphs. 
\subsection{A key lemma on random 1-regular graph}  We will see in \eqref{ehs}  that the generating function of the number of disagreeing edges   plays a central role in  the display of  partition function. Thanks to \eqref{hebn}, we only need study the generating functions of the number of disagreeing edges in random 1-regular graphs. For $k\leq m$, define 
\begin{equation} \label{dfog}
g(\beta,k,m):= \E \Big(\exp \big(-2 \beta X(k,m)  \big) \Big).
\end{equation}
The asymptotic behavior of $g(\beta,k,m)$  is described in the following lemma.
\begin{lem} \label{lgb}
For all $\beta \geq 0$,  there exists a positive constant $C=C(\beta)$, such  that for all $m$ large enough the following assertions hold.
\begin{itemize}
\item[(i)] For all $0 \leq k \leq \ell \leq m$,
 \[ \Big|[\log g(\beta, k,m)- m F(k/m)] - [\log g(\beta, \ell,m)- m F(\ell/m)] \Big | \leq \frac{C|k-\ell|}{m}.
\] 
\item[(ii)] We have
\[\max_{0 \leq k \leq m} \,\,\vline\frac{\log g(\beta,k,m)}{m}- F(k/m)  \, \vline \, \leq \frac{C}{m}, \]
\end{itemize}
with  $F(t)  $ as in Theorem \ref{ttd}.  
\end{lem}
\begin{proof} We  observe that $g(\beta,0,m)=1$ and $F(0)=0$. Hence, (ii) is a direct consequence of (i).  We first claim that  to prove  (i), it suffices to show  
\begin{equation} \label{lmn}
\textrm{(i) holds for all } 0 \leq k \leq \ell \leq [m/2].
\end{equation}
Indeed, we observe that for all $0 \leq k \leq m$,
\[X(k,m)\mathop{=}^{(\kD)}X(m-k,m).\]
Thus 
\begin{equation} \label{gdx}
g(\beta, k, m) = g(\beta, m-k,m).
\end{equation}
Moreover,  we have $F(t)=F(1-t)$ for all $t \in [0,1]$. Hence for all $k \leq m$,
\begin{equation} \label{fdx}
F\left(\frac{k}{m} \right) = F\left( \frac{m-k}{m}\right). 
\end{equation}
Combining \eqref{gdx} and \eqref{fdx}, we get that for $0 \leq k \leq [m/2] < \ell \leq m$,
\begin{eqnarray*}
&&\Big|[\log g(\beta, k,m)- m F(k/m)] - [\log g(\beta, \ell,m)- m F(\ell/m)] \Big | \\
&=& \Big|[\log g(\beta, k,m)- m F(k/m)] - \left[\log g(\beta, m- \ell,m)- m F \left(\frac{m-\ell}{m}\right) \right] \Big |\\
&\leq& \frac{C|k-(m-\ell)|}{m} \leq \frac{C|k-\ell|}{m},
\end{eqnarray*} 
by using \eqref{lmn} for $0 \leq k, m- \ell \leq  [m/2]$. Hence (i) holds for $0 \leq k \leq [m/2] < \ell \leq m$. Similarly, we can also prove that  (i) holds for $[m/2] \leq k \leq \ell \leq m$, and thus (i) follows. 

\vspace{0.2 cm}
\noindent We now prove  \eqref{lmn}. The demonstration of \eqref{lmn} is long and divided into four parts: recursive formula for $g(k,m)$;   reduced  sequence of $g(k,m)$; approximation of the reduced sequence, and conclusion.

\noindent {\bf I. Recursive formula}.   We claim that for all $k \leq [m/2]$,
\begin{equation} \label{xkn}
X(k,m) \mathop{=}^{(\kD)}  
\begin{cases}
X(k-2,m-2) & \textrm{ with prob. } (k-1)/(m-1)\\
1+ X(k-1,m-2) & \textrm{ with prob. } (m-k)/(m-1).
\end{cases}
\end{equation}
Indeed, we remind the construction of the random $1$-regular graph:  to each vertex in  $\bar{V}_m$ we attach  an half-edge, then we pair these half-edges uniformly. Let us denote by $\bar{\kU}_k$ (resp. $\bar{\kU}_k ^c$)  the set of half-edges that incident to $\bar{U_k}$ (resp. $\bar{U}_k^c$).   Suppose that we start the procedure of pairing half-edges with an element in $\bar{\kU}_k$, say $h_1$.  Then there are two possibilities. First,  with probability $(m-k)/(m-1)$, the half-edge $h_1$ is paired with an element in  $\bar{\kU}_k^c$.  This paring  gives an edge between $\bar{U}_k$ and $\bar{U}_k^c$. After this step, there remains $m-2$ half-edges including $k-1$ ones belonging to $\bar{\kU}_k$. Hence $X(k,m)$ has the same law as $1+X(k-1,m-2)$. Secondly,  with probability $(k-1)/(m-1)$, the  half-edge $h_1$ is paired with an element in $\bar{\kU}_k$, and  that does not give an edge between $\bar{U}_k$ and $\bar{U}_k^c$. Thus after this step,   $X(k,m)$ has the same law as $X(k-2,m-2)$.  

\vspace{0.2 cm}
Now applying  \eqref{xkn}, we obtain 
\begin{eqnarray} \label{trh1}
g(\beta,k,m) &=& \en \big(e^{-2 \beta X(k,m)}\big) \notag \\
&= &  \frac{k-1 }{m-1} \en \big(e^{-2 \beta X(k-2,m-2)}\big)  + \frac{m-k}{m-1} \en \big(e^{-2 \beta [1+X(k-1,m-2)]} \big) \notag  \\
&=&  \frac{k-1}{m-1} g(\beta, k-2,m-2) + \frac{(m-k)e^{-2 \beta}}{m-1} g(\beta, k-1,m-2). 
\end{eqnarray} 
As for \eqref{xkn}, starting   with an half-edge in $\bar{\kU}_k^c$, we get 
\begin{equation*} 
X(k,m) \mathop{=}^{(\kD)}  
\begin{cases}
X(k,m-2) & \textrm{ with prob. } (m-k-1)/(m-1)\\
1+ X(k-1,m-2) & \textrm{ with prob. } k/(m-1).
\end{cases}
\end{equation*}
Hence 
\begin{eqnarray} \label{trh2}
g(\beta,k,m) &=& \frac{(m-k-1)}{m-1} g(\beta, k,m-2) + \frac{k e^{-2 \beta}}{m-1} g(\beta, k-1,m-2). 
\end{eqnarray} 
It follows from  \eqref{trh1} and \eqref{trh2}  that  
\begin{equation} \label{gkn}
g(\beta,k,m-2)= \frac{(m-2k)e^{-2 \beta}}{m-k-1} g(\beta, k-1,m-2) + \frac{k-1}{m-k-1} g(\beta, k-2,m-2).
\end{equation}
We replace $m-2$ by $m$ in \eqref{gkn} and obtain a recursive formula
\begin{equation} \label{ngkn}
g(\beta,k,m)= \frac{(m-2k+2)e^{-2 \beta}}{m-k+1} g(\beta, k-1,m) + \frac{k-1}{m-k+1} g(\beta, k-2,m).
\end{equation} 
{\bf II. Reduced sequence}. Define for all $1 \leq i\leq m$,
\[h(\beta,i,m)= \frac{g(\beta,i,m)}{g(\beta,i-1,m)}.\]
Then we have 
\begin{equation} \label{gvh}
\log g(\beta,k,m) =\sum_{i=1}^k \log h(\beta,i,m).
\end{equation}
Moreover by \eqref{ngkn}, 
\begin{equation} \label{hkn}
h(\beta,k,m)= \frac{(m-2k+2)e^{-2 \beta}}{m-k+1}  + \frac{k-1}{(m-k+1) h(\beta, k-1,m)}.
\end{equation} 
Observe that $g(\beta,0,m)=1$ and $g(\beta,1,m)=e^{-2\beta}$, since $X(0,m)=0$ and $X(1,m)=1$. Thus $h(\beta,1,m)=e^{-2\beta}$. For simplicity, we  remove the notation $\beta$ in the function $h$ and denote  
$$c= e^{-2\beta} \in (0,1).$$
 Then $h(1,m)= c$. Moreover, by replacing $k$ by $k+1$ in \eqref{hkn}, we get 
 \begin{equation} \label{nhkn}
h(k+1,m)= \frac{c(m-2k) }{m-k}  + \frac{k}{(m-k) h( k,m)}.
\end{equation} 
{\bf III. Approximation of $h( k,m)$}. By numerical analysis, we find that $h(k+1,m)$ and $h(k,m)$ are very close when $m$ tends to infinity. Hence, it is natural to expect that $h(k,m)$ is approximated by the solution of the fixed point equation
 \begin{equation*} 
\theta_k= \frac{c(m-2k) }{m-k}  + \frac{k}{(m-k) \theta_k}.
\end{equation*} 
Going further to approximate  the  sequence $h(k,m)$, we  consider the following functional equation
\begin{equation} \label{eqx}
\theta= \frac{c(1-2t) }{1-t}  + \frac{t}{\theta(1-t)}.
\end{equation} 
The positive solution of this equation is 
\begin{equation} \label{rex}
\theta=f(t):= \frac{c(1-2t)+ \sqrt{1+(c^2-1)(1-2t)^2}}{2(1-t)}.
\end{equation}
We claim the following estimates on $f(t)$ and $h(k,m)$. 
\begin{itemize}
\item[$\bullet$] For all $t \in [0,1/2]$,
\begin{equation} \label{chfm}
 c \leq f(t) \leq1. 
\end{equation} 
\item[$\bullet$] There exists a positive constant $A=A(\beta)\geq 1$, such that for all $t \in (0,1/2)$
\begin{equation} \label{fpc}
1/A \leq  f'(t) \leq A \qquad \textrm{ and }  \qquad |f''(t)| \leq A. 
\end{equation} 
\item[$\bullet$] There exists a positive constant $\varkappa$, such that for all $m$ and  $0 \leq k  \leq [m/2]$,
\begin{equation} \label{fhf}
 \Big | h(k,m) -f \left(\frac{k-1}{m}\right) \Big | \leq \frac{\varkappa}{m}.
\end{equation}
\end{itemize}
Note that the bound for $f''(t)$ in \eqref{fpc} is not used in  the proof of \eqref{lmn}, but it is  needed for the proof of \eqref{fhf}. The proof of \eqref{chfm}, \eqref{fpc} and \eqref{fhf} is long and complicated, so we put it in Appendix.

\noindent {\bf IV. Conclusion}. Assuming these claims \eqref{chfm}, \eqref{fpc}, \eqref{fhf}, we now prove \eqref{lmn}. By \eqref{chfm}  and \eqref{fhf}, we have for all $m$ large enough and $0 \leq k \leq [m/2]$,
\begin{equation} \label{hvkc}
c/2 \leq \min \{ h(k,m), f(k/m) \}.
\end{equation}
 Using the mean value theorem, we have for all $x,y >0$,
\begin{equation} \label{lab}
|\log x- \log y| \leq \frac{|x-y|}{\min \{x,  y \} }.
\end{equation}
Using \eqref{gvh},  \eqref{fhf},  \eqref{hvkc} and \eqref{lab}, we get that for all $0  \leq k \leq \ell  \leq [m/2]$,
\begin{eqnarray}
\vline \, \log g(k,m) - \log g(\ell,m) + \sum_{i=k+1}^{\ell} \log f\Big( \frac{i-1}{m} \Big)\,\vline
%&=& \vline \, \sum_{i=1}^k \big[\log h(i,m) - \log f((i-1)/m)\big] -  \sum_{i=1}^{\ell} \big[\log h(i,m) - \log f((i-1)/m) \big]\,\vline \notag  \\ 
&=& \vline \, \sum_{i=k+1}^{\ell} \big[\log h(i,m) - \log f((i-1)/m)\big] \, \vline\notag \\
 & \leq & \sum_{i=k+1}^{\ell}  \Big |  \log h(i,m) -   \log f((i-1)/m) \Big | \notag   \\
& \leq &  \frac{2}{c}\sum_{i=k+1}^{\ell } \Big| h(i,m)-f((i-1)/m) \Big | \notag\\
& \leq &  \frac{2\varkappa(\ell-k)}{cm}. \label{ltrk}
\end{eqnarray}
Similarly,
\begin{eqnarray} \label{ftf}
\vline \,\frac{\log f(i/m)}{m} - \int_{i/m}^{(i+1)/m} \log f(s) ds \, \vline &\leq& \int_{i/m}^{(i+1)/m} | \log f(i/m) - \log f(s)| ds \notag   \\
 &\leq& \frac{2}{c} \int_{i/m}^{(i+1)/m} | f(i/m) - f(s)| ds \notag   \\
&\leq & \frac{2A}{c} \int_{i/m}^{(i+1)/m} |(i/m) - s| ds \notag \\
&=& \frac{A}{m^2c}.
\end{eqnarray}
Here for the third inequality, we have used \eqref{fpc} and the mean value theorem. It follows from  \eqref{ltrk} and \eqref{ftf} that for all $0 \leq k\leq \ell \leq [m/2]$,
 \[\vline \, \log g(k,m) - \log g(\ell,m) + m \int_{k/m}^{\ell/m} \log f(s) ds \, \vline \leq \left( \frac{2\varkappa+A}{c}  \right) \left( \frac{ \ell - k}{m}\right),\]
 which proves Lemma \ref{lgb} (i).  
\end{proof}

\subsection{An auxiliary  lemma} The following  result  will be used  in the proof  of the existence of the annealed pressure.
\begin{lem} \label{lpre} The following assertions hold.
\begin{itemize}
\item[(i)] Let $G(t)$ be a continuous function on $[0,1]$. Then
\[\lim_{n \rightarrow \infty} \max \limits_{0 \leq j \leq n} G(j/n)= \max\limits_{0 \leq t \leq 1} G(t).\] 
\item[(ii)] Let $(G_n(t))_n$ be a sequence of  functions on $[0,1]$,  which converges point-wise  to a fucntion $G (t)$. Suppose that there exists a positive constant $C$ and a sequence $(\varepsilon_n)$ tending to $0$, such that for all $0 \leq s,t \leq 1$ and $n \geq 1$,
\begin{equation*} \label{gnst}
|G_n(s)-G_n(t)| \leq C|s-t| + \varepsilon_n.
\end{equation*}
 Then $G(t)$ is a Lipschitz function. Moreover,  for any continuous function $H(t)$ on $[0,1]$, we have 
 \[\lim_{n \rightarrow \infty} \max \limits_{0 \leq j \leq n} [ H(j/n) +G_n(j/n)]= \max\limits_{0 \leq t \leq 1} [H(t)+G(t)].\] 
\end{itemize}
\end{lem}
The results of  this lemma are standard in real analysis, so  we safely leave to the reader. 
\subsection{Notation}
If $f $ and $g$ are two real functions, we write $f= \mathcal{O}(g)$ if there exists a constant $C>0,$ such that $f(x) \leq C g(x)$ for all $x ;$  $f \asymp g $ if $f= \mathcal{O}(g)$ and $g= \mathcal{O}(f);$  $f=o(g)$ if $g(x)/f(x) \rightarrow 0$ as $x \rightarrow \infty$.

\noindent Let $(f(j,n))_{1 \leq j \leq n }$ and $(g(j,n))_{1 \leq j \leq n}$ be two sequences of real numbers. The notion $f(j,n)= \kO(g(j,n))$ (or $f(j,n)= o(g(j,n))$) is taken uniformly in all $j \leq n$.

\iffalse
 More precisely, 
$f(j,n)= \kO(g(j,n))$ means that there exists a positive constant $C$, such that for all $n$ large enough and $j \leq n$,
\[\Big |\frac{f(j,n)}{g(j,n)} \Big | \leq C.\]
On the other hand, $f(j,n)= o(g(j,n))$ means that for any $\varepsilon >0$, there exists $n_{\varepsilon}$ such that for all $n\geq n_{\varepsilon}$ and $j \leq n$,
 \[\Big |\frac{f(j,n)}{g(j,n)} \Big | \leq \varepsilon.\]
 \fi
\noindent For any real number $x$, let $[x]$ denote the integer part of $x$.
\section{The annealed pressure}
The first step (which is one of  the most important steps) in studying the Ising model is the task of understanding the partition function and the pressure. As mentioned in the introduction, we will write the Hamiltonian in term of the number of disagreeing edges. Then using  the symmetry of  random regular graphs, we can investigate the annealed pressure.   Let us be more precise now.  

We fix an integer $d \geq 2$. Then for any positive integer $n$, we consider the random $d$-regular graph whose the vertex set is $V_n= \{v_1, \ldots, v_n\}$.   For any spin configuration $\sigma \in \Omega_n$,  define 
\[\sigma_+= \{v_i: \sigma_i=1\} \qquad \textrm{and} \qquad \sigma_-= \{v_i: \sigma_i=-1\}.\]
Then
\begin{eqnarray*}
\sum_{i=1}^n \sigma_i &= &2|\sigma_+| -n, \\
\sum_{i\leq j} k_{i,j} \sigma_i \sigma_j &=& (dn)/2 - 2e(\sigma_{+}, \sigma_{-}),
\end{eqnarray*}
where  
$$e(\sigma_{+}, \sigma_{-}) = \#\{\textrm{ edges between   $\sigma_{+}$ and $ \sigma_{-}$}\}.$$ 
Therefore
\begin{equation} \label{hnsg}
H_n(\sigma) =   \left(B-\frac{\beta d}{2} \right)n + 2\beta e(\sigma_{+}, \sigma_{-}) - 2 B |\sigma_{+}|.
\end{equation}
Thus 
\begin{equation} \label{ehs}
\E\left(e^{-H_n(\sigma)}\right)= e^{ \left(\frac{\beta d}{2} - B\right)n} \E \left( e^{ - 2\beta e(\sigma_{+}, \sigma_{-})} \right)  e^{2 B |\sigma_{+}|}.
\end{equation}
By \eqref{hebn}, if $|\sigma_+|= |\sigma'_+|$ then 
\[  e(\sigma_{+}, \sigma_{-}) \mathop{=}^{(\kD)} e(\sigma'_{+}, \sigma'_{-})  \mathop{=}^{(\kD)} X(d|\sigma_+|, dn)). \]
   Hence 
\begin{eqnarray}
 && \sum_{\sigma \in \Omega_n} \E \left( e^{ - 2\beta e(\sigma_{+}, \sigma_{-})} \right)  e^{2 B |\sigma_{+}|} = \sum_{j=0}^n e^{2 Bj} \sum_{\substack{\sigma \in \Omega_n \\ |\sigma_+| =j}} \E \left( e^{ - 2\beta e(\sigma_{+}, \sigma_{-})} \right)    \notag \\
 &=&  \sum_{j=0}^n \binom{n}{j} e^{2Bj}\E \left(e^{-2 \beta X(dj,dn)}\right) = \sum_{j=0}^n \binom{n}{j} e^{2Bj}g(\beta,dj,dn),
\end{eqnarray}
with  $g(\beta,k,m)$ defined as in \eqref{dfog} for all $k\leq m$.
 Therefore
\begin{equation} \label{ezn}
\E(Z_n(\beta,B))= e^{ \left(\frac{\beta d}{2} - B\right)n} \times \sum_{j=0}^n \binom{n}{j} e^{2Bj} g(\beta,dj,dn). 
\end{equation}

\vspace{0.2 cm}
\noindent {\it Proof of Theorem \ref{ttd} (i).} 
By \eqref{ezn}, we have
\begin{equation*} \label{eznbb}
\frac{1}{n} \log\E(Z_n(\beta,B))= \frac{\beta d}{2} -B + \max_{0 \leq j \leq n} \left[\frac{\log \binom{n}{j}}{n}+ 2B \frac{j}{n} + \frac{\log g(\beta,dj,dn)}{n}\right] +o(1).
\end{equation*}
On the other hand, it follows from  Stirling's formula that  
\begin{equation*} \label{bnjm}
\frac{\log \binom{n}{j}}{n}= \frac{j}{n} \log \left(\frac{n}{j}\right) +\frac{n-j}{n} \log \left(\frac{n}{n-j}\right) +o(1).
\end{equation*} 
Combining the last two equations   and Lemma \ref{lgb} (ii), we obtain
\begin{equation} \label{eznbb}
\frac{1}{n} \log\E(Z_n(\beta,B))= \frac{\beta d}{2} -B + \max_{0 \leq j \leq n}  L(j/n) + o(1),
\end{equation}
where $L(t)$ is a continuous function on $[0,1]$ defined by 
\[L(t)= -t \log (t) + (t-1) \log (1-t) + 2 Bt + d F(t). \]
Now, the result follows from \eqref{eznbb} and Lemma \ref{lpre} (i). \hfill $\square$

\vspace{0.3  cm}
An explicit formula for the function $F(t)$ is given in the following lemma.
\begin{lem} \label{lofl} For $t\leq 1/2$, we have 
\begin{eqnarray*}  
F(t) &= & t \log f(t) +  \frac{1}{2} \log (1-t) + \ \frac{1}{2}\log (1+e^{-2 \beta})   + \frac{1}{2}\log \left[ 1+ \frac{e^{-2 \beta}(2t-1)}{(1-t)(f(t)+1)} \right].
\end{eqnarray*}
For  $t\in (1/2,1)$, we have $F(t)=F(1-t)$.
\end{lem}
 The quenched pressure $\tilde{\psi}(\beta, B)$ has   been determined in \cite[Theorem 2.4]{DM}. The equality between the annealed and quenched pressures is established in the following proposition.  
\begin{prop} \label{eopp} For all $\beta >0$ and $B \in \R$, we have $\psi(\beta, B)=\tilde{\psi}(\beta, B)$. In particular, when $d=2$, 
\[\psi(\beta, B) = \beta + \log \left( \cosh (B) + \sqrt{ \sinh ^2 (B) + e^{-4 \beta}} \right), \]
which agrees with the result  obtained in \cite{GGHPb}. 
\end{prop}
The proof of Lemma \ref{lofl} and Proposition \ref{eopp} is put in Appendix.
\section{The annealed magnetization and the strong law of large number}
In this section, we prove the existence of the annealed magnetization and Theorem \ref{tll} following the strategy mentioned in the introduction. 

\vspace{0.2 cm}
\noindent {\it Proof of Theorem \ref{ttd} (ii)}.  We state the following claims which we prove below.
\begin{itemize}
\item[$\bullet$] Claim 1. For any $\beta \geq 0$, the function $\psi(\beta, \cdot)$ is differentiable at every point $B \neq 0$.
\item[$\bullet$] Claim 2. For any $d\geq 3$,
\begin{equation*}
\beta_c= \atanh(1/(d-1))= \frac{1}{2} \log \left( \frac{d}{d-2} \right).
\end{equation*}
Moreover, for any $\beta \in (0, \beta_c)$, the function $\psi(\beta, \cdot)$ is differentiable at $B=0$.
\end{itemize}
Assuming these claims, Theorem \ref{ttd} (ii) follows. Indeed,  using similar arguments as in the proof of  Theorem 1.1 (ii) in \cite{GGHPb}, we can  show  that for all $(\beta,B) \in \kU$, the annealed magnetization $(M_n(\beta,B))$ converges to
\begin{equation} \label{fmbb}
\kM(\beta,B): = \frac{\partial \psi(\beta,B)}{\partial B}.
\end{equation}
This together with the claims 1 and 2  imply Theorem \ref{ttd} (ii). \hfill $\square$

\vspace{0.2 cm}
{\it Proof of Claim 1}.   We consider here the case $B>0$, the other one can be handled similarly. We first define some functions on $[0, 1]$:
\begin{eqnarray} \label{ihl}
I(t) &=& (t-1) \log (1-t) - t \log t, \notag \\
H(t) &=& I(t) + d\int_0^{u(t)} \log f(s) ds, \\
L(t) &=& H(t) + 2Bt. \notag
\end{eqnarray}
 By Theorem \ref{ttd} (i),  
\begin{equation} \label{dpbb}
\psi(\beta,B)= (\beta d)/2 -B + \max_{0 \leq t \leq 1} L(t).
\end{equation}
Observe that $H(t)=H(1-t)$ for all $t\in [0,1]$, and $Bt \leq B(1-t)$ for all $t\leq 1/2$. Hence  $L(t)=H(t) +2Bt$ attains the maximum at a point in $[1/2,1]$. We consider the derivative of $L(t)$ on  $[1/2,1]$:
\begin{equation*} \label{hpb}
L'(t)=H'(t)+2B= \log\left(\frac{1-t}{t}\right) - d  \log f(1-t) +2B.
\end{equation*}
We have $L'(1/2)=2B>0$ and $L'(1^-) = - \infty$, so the maximum point of $L(t)$ is a solution of the equation 
\begin{equation} \label{ybb}
L'(t)=\log\left(\frac{1-t}{t}\right) - d  \log f(1-t) +2B=0.
\end{equation}

\vspace{0.2 cm}
 Claim 1$^*$. The equation \eqref{ybb} has a unique solution $t_*$ in $(1/2,1)$,  and  $L''(t_*) \neq 0$. 
 
 \vspace{0.2 cm}
\noindent Assuming  this claim,  we can deduce from the implicit function theorem that the function $t_*$ is differentiable with respect to  $B$. Thus the function $\psi(\beta,\cdot)$ is also differentiable and Claim 1 follows. Moreover,
\begin{eqnarray} \label{kmbb}
\frac{\partial}{\partial B} \psi(\beta,B) = -1 + H'(t_*) \frac{\partial t_*}{\partial B} + 2t_*+2B \frac{\partial t_*}{\partial B} = -1 + 2t_*.
\end{eqnarray}
Now we prove Claim 1$^*$. Since  $L'(1/2)=2B >0$ and $L'(1^{-})=-\infty$, the function  $L'(t)$ has at least one root in $(1/2,1)$. Suppose that $L'(t)$ has more than one root in $(1/2,1)$. Then  $L''(t)$ has at least two roots in $(1/2,1)$. We consider the following equation in $(1/2,1)$ 
\begin{equation} \label{eqlt}
L''(t)= \frac{-1}{t(1-t)} - \frac{d x'(t)}{x(t)} =0,
\end{equation}
where 
$x(t)=f(1-t)$. Since  $f(t)$ satisfies \eqref{eqx}, we have  
\[x(t)= \frac{c(2t-1)}{t}+ \frac{1-t}{tx(t)},\]
with 
\[c=e^{-2\beta} \in (0,1).\]
%Therefore
%\[x'(t)=\frac{c}{t^2} -\frac{1}{t^2x(t)} - \frac{x'(t)}{x(t)^2} \frac{1-t}{t}.\]
%Hence
After some computation, we get
\begin{equation} \label{xpt}
\frac{x'(t)}{x(t)}= \frac{c x(t)-1}{t \big[c(2t-1)x(t)+2-2t \big]}.
\end{equation}
Using this and \eqref{eqlt}, we obtain
\begin{equation} \label{ypb}
L''(t)= \frac{-d(1-t)(cx(t)-1) -c(2t-1)x(t)+2-2t}{t(1-t)[c(2t-1)x(t)+2-2t]}.
\end{equation}
Hence 
\begin{eqnarray*}
L''(t)=0 & \Leftrightarrow & d(1-t)(1-cx(t))=c(2t-1)x(t)+2-2t \\
& \Leftrightarrow & x(t) = \frac{(d-2)(1-t)}{c(d(1-t)+2t-1)} \\
& \Leftrightarrow & \frac{c(2t-1)+ \sqrt{1+(c^2-1)(2t-1)^2}}{2t} = \frac{(d-2)(1-t)}{c(d(1-t)+2t-1)} \\
& \Leftrightarrow &  \sqrt{1+(c^2-1)(2t-1)^2} = \frac{2t(d-2)(1-t)}{c(d(1-t)+2t-1)} -c(2t-1),
\end{eqnarray*}
from which it follows that  
\begin{eqnarray*}
  & &1+(c^2-1)(2t-1)^2 =\frac{4t^2(d-2)^2(1-t)^2}{c^2(d(1-t)+2t-1)^2} +c^2(2t-1)^2 - \frac{4t(d-2)(1-t)(2t-1)}{(d(1-t)+2t-1)} \notag \\
  &\Leftrightarrow & \quad  4t-4t^2 = \frac{4t^2(d-2)^2(1-t)^2}{c^2(d(1-t)+2t-1)^2} - \frac{4t(d-2)(1-t)(2t-1)}{(d(1-t)+2t-1)}  \notag \\
  &\Leftrightarrow & \quad  c^2(d(1-t)+2t-1)^2 = t(1-t)(d-2)^2 - c^2(d-2)(2t-1)}{(d(1-t)+2t-1),
  \end{eqnarray*}
  or equivalently
  \begin{equation}
 c^2 \big[(d-2)^2(t-t^2)+d-1\big] = t(1-t)(d-2)^2. \label{ybph}
  \end{equation}
Since $d\geq 3$, the equation \eqref{ybph} is equivalent to
\begin{equation} \label{eqt}
      t^2 -t+ \frac{c^2(d-1)}{(d-2)^2(1-c^2)}=0.
\end{equation} 
Observe that the sum of two solutions of \eqref{eqt} is $1$. Hence \eqref{eqt} has  at most one solution in $(1/2,1)$.  Therefore $L''(t)$  has at most  one root in $(1/2,1)$. Hence the equation $L'(t)=0$ has a unique solution in $(1/2,1)$, say $t_*$. Now we show $L''(t_*) \neq 0$ by contradiction. Suppose  that $L''(t_*) = 0$. Then $t_*$ must be a solution of   \eqref{eqt}.  Hence
\begin{eqnarray*} 
\frac{c^2(d-1)}{(d-2)^2(1-c^2)} =t_*-t_*^2 < \frac{1}{4}.
\end{eqnarray*}
Thus 
\begin{equation} \label{ced}
  c < \frac{d-2}{d}.
\end{equation}
Since $L'(1/2)>0$ and $L'(t_*)=0$,  there exists $u \in (1/2, t_*)$, such that $L''(u)<0$. On the other hand, by \eqref{ypb} and \eqref{ced},
 \[L''(1/2)= -4 - 2d(c-1) >0.\]
Since $L''(u) L''(1/2) <0$, the function $L''(t)$ has a root in $(1/2,u)$. Hence $L''(t)$ has at least two roots in $(1/2,1)$, which leads a contradiction. Therefore $L''( t_*) \neq 0$ and Claim 1$^*$ follows. \hfill $\square$
 
 \vspace{0.2 cm}
{\it Proof of Claim 2}.  
 Claim 2 is a direct consequence of the following claims.  
 \begin{itemize}
 \item[$\bullet$] Claim 2a. If $\beta >  \textrm{atanh}(1/(d-1))$ then
\begin{equation*}
\lim \limits_{B \searrow 0} \kM(\beta,B) =-1 + 2t_{+} >0,
\end{equation*}
 where $t_+$ is the unique root in $(1/2,1)$ of the function $H'(t)$. 
 \item[$\bullet$] Claim 2b. If $0< \beta <  \textrm{atanh}(1/(d-1))$ then $H'(t)$ is strictly decreasing on $(0,1)$ and   has a unique root $t_0=1/2$. Moreover, the function $\psi(\beta, \cdot)$ is differentiable at $B=0$  and 
 \[\lim_{B \searrow 0} \kM(\beta,B) =\frac{\partial}{\partial B} \psi(\beta,B) \Big|_{B=0} =0.\]
 \end{itemize}
We first prove Claim 2a. Observe that $H'(1/2)=0$ and $H'(1^-)=-\infty$. Moreover,  by \eqref{ypb} 
 \[H''(1/2)= -4 - 2 d(e^{-2\beta}-1) >0,\]
 since 
 \[\beta > \textrm{atanh}(1/(d-1))=\frac{1}{2} \log \left( \frac{d}{d-2} \right).\]
Therefore  $H'(t)$ has at least one root in $(1/2,1)$. Using the same arguments as in the proof of Claim 1$^*$, $H'(t)$ has at most one root in $(1/2,1)$. Thus it  has a unique  root $t_+$ in $(1/2,1)$. Moreover, $H'(t)>0$ for all $t \in (1/2,t_+)$ and $H'(t)<0$ for all $t\in (t_+,1)$.   

On the other hand, $0=L'(t_*)=H'(t_*)+2B$. Hence $H'(t_*)=-2B<0$ when $B>0$. Therefore $t_* > t_+$ for all $B>0$. In addition, $\lim \limits_{B \searrow 0} t_* = t_+$. Hence by \eqref{kmbb}, we have
\[\lim \limits_{B \searrow 0} \kM(\beta,B)= -1 +2 t_+ >0,\]
 which implies Claim 2a. 
 
 We now prove Claim 2b. Assume that 
\begin{equation} \label{ebed}
0< \beta < \frac{1}{2} \log \left( \frac{d}{d-2} \right).
\end{equation}
We first show that $H''(t) <0$ for all $t \in (0,1)$. We consider here the case $t\geq 1/2$, the other one is similar.    Using  \eqref{ypb} and  the same calculation as for \eqref{ybph}, we have 
\begin{eqnarray*}
& & H''(t)=L''(t)<0 \, \Leftrightarrow \, d(1-t)(1-cx(t))<e^{-2 \beta }(2t-1)x(t)+2-2t \\
& \Leftrightarrow &  \sqrt{1+(e^{-4 \beta }-1)(2t-1)^2} > \frac{2t(d-2)(1-t)}{e^{-2 \beta }(d(1-t)+2t-1)} -e^{-2 \beta }(2t-1),
\end{eqnarray*}
from which it follows that  
\begin{equation} \label{ibc}
 e^{-4 \beta }\big[(d-2)^2(t-t^2)+d-1 \big] > t(1-t)(d-2)^2.
\end{equation}
Under the condition \eqref{ebed},  the  inequality \eqref{ibc} is a consequence of the following
 \begin{eqnarray*}
      (d-2)^2(t-t^2)+d-1 &>& t(1-t)d^2 \\
 \Leftrightarrow \quad     1&>& 4t(1-t),  
\end{eqnarray*}
which holds for all $t\in(0,1) \setminus \{1/2\}$.  For $t=1/2$,
 \[H''(1/2)= -4 - 2 d(e^{-2\beta}-1) <0,\]
by \eqref{ebed}. In conclusion, $H''(t) <0$ for all $t \in (0,1)$ and thus  $H'(t)$ is strictly decreasing and has a unique zero at $t=1/2$. Now applying the implicit function theorem for the function $L'(t)$, we get that  $t_*$, the solution of the equation $L'(t)=0$, is  differentiable with respect to $B$ at  $0$. Thus  the function $\psi(\beta, \cdot)$ is also differentiable at $B=0$ and  
\begin{eqnarray}
\lim \limits_{B \searrow 0} \kM(\beta,B) = \frac{\partial}{\partial B} \psi(\beta,B) \Big|_{B=0} = \lim \limits_{B \rightarrow 0} -1 + 2t_* = -1+2t_0=0.
\end{eqnarray}
This implies the claim 2b. 
 \hfill $\square$

\vspace{0.2 cm}
{\it Proof of Theorem \ref{tll}}. As mentioned in the introduction,  the  exponentially strong law large numbers for the magnetization follows from the differentiability of the pressure $\psi(\beta, B)$ with respect to $B$, by using the same arguments  in the proof of  \cite[Theorem 1.2]{GGHPb}. \hfill $\square$

\section{The annealed susceptibility and the central limit theorem}
We have shown  that for all $(\beta,B) \in \kU$,
\[\frac{\partial}{\partial B} \psi(\beta,B) = -1 +2 t_*,\]
where  $t_*$ is the solution of the equation 
\[L'(t)=H'(t)+2B =0,\]
with $L(t)$ and $H(t)$ as in \eqref{ihl}. Moreover, we showed that $t_*$ is a differentiable function with respect to $B$. Hence
\[H''(t_{*}) \frac{\partial t_{*}}{\partial B}  +2 =0,\]
and thus 
\[ \frac{\partial t_{*}}{\partial B}   = \frac{-2}{H''(t_{*})}.\]
Therefore 
\begin{equation} \label{cbd}
\chi(\beta,B):=\frac{\partial^2}{\partial B^2} \psi(\beta,B) = 2 \frac{\partial  t_*}{\partial B} = \frac{-4}{H''(t_{*})}.
\end{equation}
Let us recall the  definition of  the sequence of  cumulant generating functions  
\[c_n(t)=\psi_n(\beta,B+t)-\psi_n(\beta,B).\]
\begin{lem} \label{lcgf}
Suppose that $(\beta,B) \in \kU$. Then for any positive constant $t$ and any sequence $(t_n)$ satisfying $t_n \leq t/ \sqrt{n}$, we have
\begin{equation*} \label{cnht}
c_n''(t_n) \rightarrow \chi(\beta,B).
\end{equation*}
\end{lem} 

\vspace{0.2 cm}
\noindent {\it Proof of Theorem \ref{ttd} (iii)}.  The result is a consequence of Lemma \ref{lcgf} with $t_n \equiv 0$. \hfill $\square$

\vspace{0.2 cm}

\noindent{\it Proof of Theorem \ref{tclt}}. As mentioned in the introduction,  the central limit theorem is  a consequence of Lemma \ref{lcgf} by applying the same arguments  as in the proof of \cite[Theorem 1.6]{GGHPb} and \cite[Theorem A.8.7]{E}.  \hfill $\square$

\vspace{0.2 cm}
\noindent {\it Proof of Lemma \ref{lcgf}}. We consider here the case $B\geq 0$, the other one can be handled similarly. Thanks to \eqref{cbd}, we only need to show that  for any positive constant $t$ and any sequence $(t_n)$ satisfying $t_n \in [0,t/\sqrt{n}]$,
\begin{equation} \label{cnts}
c_n''(t_n) \rightarrow \frac{-4}{H''(t_{*})}.
\end{equation}
It follows from \eqref{ezn} that for all $s>0$,
\begin{equation} \label{cppt}
c_n''(s)= \frac{\partial^2}{\partial B^2} \log Z_n(\beta,B+s) =\frac{4}{n} \left( \frac{T_{2,n}(s)}{T_n(s)} - \left( \frac{T_{1,n}(s)}{T_n(s)} \right)^2\right),
\end{equation}
where
\begin{eqnarray*}
T_n(s)&=& \sum_{j=0}^n \binom{n}{j} e^{2(B+s)j} g(\beta,dj,dn) \\
T_{1,n}(s)&=& \sum_{j=0}^n \binom{n}{j} e^{2(B+s)j} g(\beta,dj,dn)j \\
T_{2,n}(s)&=& \sum_{j=0}^n \binom{n}{j} e^{2(B+s)j} g(\beta,dj,dn)j^2.
\end{eqnarray*} 
Let us define
\[j_{*}=[nt_*].\]
We will show that the values of $T_{n}(s), T_{1,n}(s), T_{2,n}(s)$ are concentrated around the $j_{*}$ th term of each sum if $s=\kO(1/\sqrt{n})$.  We fix a positive constant $t$ and a sequence $(t_n)$ satisfying $t_n \in [0,t/ \sqrt{n}]$.   Define 
\begin{eqnarray*}
\bar{T}_n(t_n)&=& \sum_{|j-j_*|\geq n^{5/6}} x_j  \hspace{1.2 cm} \textrm{and} \hspace{1.2 cm}  \hat{T}_n(t_n) = \sum_{|j-j_*|< n^{5/6}} x_j  \\ % \binom{n}{j} e^{2(B+t)j} g(\beta,dj,dn) \\
\bar{T}_{1,n}(t_n)&=& \sum_{|j-j_*|\geq n^{5/6}} jx_j \hspace{1.06 cm} \textrm{and} \hspace{1.06 cm} \hat{T}_{1,n}(t_n) =  \sum_{|j-j_*|< n^{5/6}} jx_j \\% \binom{n}{j} e^{2(B+t)j} g(\beta,dj,dn)j \\
\bar{T}_{2,n}(t_n)&=& \sum_{|j-j_*|\geq n^{5/6}} j^2 x_j \hspace{1cm} \textrm{and} \hspace{1cm} \hat{T}_{2,n}(t_n)= \sum_{|j-j_*|< n^{5/6}} j^2 x_j, %\binom{n}{j} e^{2(B+t)j} g(\beta,dj,dn)j^2.
\end{eqnarray*}
where 
\begin{eqnarray*}
x_{j}&= &\binom{n}{j} e^{2(B+t_n)j} g(\beta,dj,dn).
\end{eqnarray*}
To prove  \eqref{cppt}, it suffices to  show  that 
\begin{equation} \label{tvtm}
\vline \, \left[ \frac{T_{1,n}(t_n)}{T_n(t_n)} - \left(\frac{T_{2,n}(t_n)}{T_n(t_n)} \right)^2 \right] - \left[\frac{\hat{T}_{1,n}(t_n)}{\hat{T}_n(t_n)} - \left(\frac{\hat{T}_{2,n}(t_n)}{\hat{T}_n(t_n)} \right)^2  \right] \, \vline \leq \frac{4}{n^2},
\end{equation}
and
\begin{equation} \label{thtm}
\frac{4}{n} \left(\frac{\hat{T}_{2,n}(t_n)}{\hat{T}_n(t_n)} - \left( \frac{\hat{T}_{1,n}(t_n)}{\hat{T}_n(t_n)}\right)^2 \right) \longrightarrow \frac{-4}{H''(t_{*})} \hspace{2 cm} \textrm{as } \, n \rightarrow \infty.
\end{equation}
Before proving \eqref{tvtm} and \eqref{thtm}, we make a comparison between $x_{j_*}$ and the other terms. Using  Stirling's formula, we have 
\begin{equation*} \label{bnj}
\binom{n}{j}= \left(\frac{1}{\sqrt{2 \pi}}+\kO(n^{-1})\right)\sqrt{\frac{n}{j(n-j)}} \exp \left(n I \left( \frac{j}{n} \right) \right),
\end{equation*}
where the function $I(t)$ is defined in \eqref{ihl}.  Thus   
  \begin{eqnarray} \label{tjj}
  \frac{x_j}{x_{j_*}}&=& (1+o(1)) \sqrt{\frac{j_*(n-j_*)}{j(n-j)}} \exp \Big(  2(B+t_n)(j-j_*)+ n I(j/n)  - n I(j_*/n) \notag \\
      &&  \hspace{8.0 cm} + \log g(\beta, dj,dn) - \log g(\beta,dj_*,dn) \Big) \notag\\
&=& (1+o(1)) \sqrt{\frac{j_*(n-j_*)}{j(n-j)}} \exp \Big(  2t_n(j-j_*) + n\left[ I(j/n)  +dF(j/n) +2Bj/n \right] \notag \\
&&  \hspace{2.4 cm} -n\left[ I(j_*/n) +dF(j_*/n) +2Bj_*/n\right] + \left[\log g(\beta,dj,dn)-ndF(j/n) \right] \notag\\
&&  \hspace{8.4 cm}  - \left[\log g(\beta,dj_*,dn)-ndF(j_*/n) \right] \Big) \notag \\
&=& (1+o(1)) \sqrt{\frac{j_*(n-j_*)}{j(n-j)}} \exp \Big(  2t_n(j-j_*) + n\left[L(j/n)-L(j_*/n)\right] \\
&& \hspace{2.4 cm} + \left[\log g(\beta,dj,dn)-ndF(j/n) \right] - \left[\log g(\beta,dj_*,dn)-ndF(j_*/n) \right] \Big). \notag
  \end{eqnarray}
    We have some observations on the function $L(t)$ and its  derivatives. Since $L(t)$ attains the maximum at a unique point $t_* \in (0,1)$, 
\begin{itemize}
\item[(O1)]  $L'(t_*)=0$ and $L''(t_*)<0$, 
\item[(O2)]   there exists a positive constant $\varepsilon $, such that  for all $\epsilon \leq \varepsilon$,
\[\max_{|t-t_*|\geq \epsilon} L(t) = \max \{L(t_*-\epsilon ),L(t_*+\epsilon)\}.\]
\item[(O3)] For  $\delta = (1- t_*)/2$, the functions $|L'(t)|,|L''(t)|,|L'''(t)|$ are uniformly bounded in $(t_*-\delta, t_*+ \delta)$.
\end{itemize}
\noindent {\bf  I. Proof of \eqref{tvtm}}.
For  $n$ large enough (such that $n^{-1/6} \leq \varepsilon$ as in (O2)), we have for all $|j-j_*|\geq n^{5/6}$,
 \begin{equation} \label{sach}
 L(j/n)-L(j_*/n) \leq \max \{L(j_*/n + n^{-1/6})- L(j_*/n), L(j_*/n - n^{-1/6})- L(j_*/n) \}.
\end{equation}  
Using (O3) and Taylor's theorem, we get
\begin{eqnarray*}
L(j_*/n \pm n^{-1/6})- L(j_*/n) = \pm  n^{-1/6}L'(j_*/n) + n^{-1/3}L''(j_*/n)/2 + \kO(n^{-1/2}).
\end{eqnarray*}
Similarly,
\begin{eqnarray} \label{lpjn}
L'(j_*/n)=L'(t_*) + \kO(|(j_*/n)-t_*|)  =  \kO(1/n),
\end{eqnarray}
since $L'(t_*)=0$ and $|(j_*/n)-t_*| \leq 1/n$. Therefore
\begin{equation} \label{nljj}
n \big(L(j_*/n \pm n^{-1/6})- L(j_*/n)\big ) = n^{2/3}L''(j_*/n)/2 +o(n^{2/3}).
\end{equation}
On the other hand,  since $L''(t_*) <0$ and the sequence $(j_*/n)$ converges to $t_*$,  for all $n$ large enough
\begin{equation*} \label{lpj}
L''(j_*/n) \leq L''(t_*)/2.
\end{equation*}
Combining this with  \eqref{sach}, \eqref{nljj} gives  that for all $|j-j_*|\geq n^{5/6}$,
\begin{equation} \label{nlj}
n \big(L(j/n)-L(j_*/n)\big) \leq n^{2/3} L''(j_*/n)/8. 
\end{equation}
We now turn back to the formula \eqref{tjj}. Observe that  for all $j \leq n$,
  \begin{equation} \label{jj}
   \sqrt{\frac{j_*(n-j_*)}{j(n-j)}}  \leq \sqrt{n}.
  \end{equation}
On the other hand,  by Lemma \ref{lgb} (ii), for all $j \leq n$,
  \begin{equation} \label{gjj}
  |\log g(\beta,dj,dn)-ndF(j/n)| = \kO(1). 
  \end{equation}
Since $t_n \leq t/\sqrt{n}$, we have 
  \begin{equation} \label{jjt}
  t_n(j-j_*) = \kO(\sqrt{n}).
  \end{equation}
It follows from  \eqref{tjj}, \eqref{nlj},  \eqref{gjj}, \eqref{jjt} that  for $n$ large enough and  $|j-j_*|\geq n^{5/6}$, 
\begin{equation} \label{xjsj}
x_j \leq x_{j_*} \exp \big(n^{2/3}L''(t_*)/8+ \kO(\sqrt{n})\big)  \leq x_{j_*} n^{-7},
\end{equation} 
 since $L''(t_*)<0$. Therefore
\begin{equation*} \label{ngt}
\bar{T}_n(t_n), \bar{T}_{1,n}(t_n), \bar{T}_{2,n}(t_n) \leq \frac{x_{j_*}}{n^4} \leq \frac{\hat{T}_n(t_n)}{n^4}. 
\end{equation*}
On the other hand,
\begin{equation*}
\hat{T}_{1,n}(t_n), \hat{T}_{2,n}(t_n) \leq n^2 \hat{T}_n(t_n).
\end{equation*}
Hence
\begin{eqnarray*} \label{tbtm}
\vline \frac{T_{1,n}(t_n)}{T_n(t_n)} - \frac{\hat{T}_{1,n}(t_n)}{\hat{T}_n(t_n)} \vline &=& \vline \frac{\bar{T}_{1,n}(t_n)+\hat{T}_{1,n}(t_n)}{\bar{T}_{n}(t_n)+\hat{T}_{n}(t_n)} - \frac{\hat{T}_{1,n}(t_n)}{\hat{T}_n(t_n)} \vline \notag \\
& \leq & \frac{\hat{T}_n(t_n) \bar{T}_{1,n}(t_n)+\hat{T}_{1,n}(t_n) \bar{T}_{n}(t_n)}{\hat{T}_n(t_n)^2} \leq \frac{1+n^2}{n^4}.
\end{eqnarray*}
Similarly, we also have 
\begin{eqnarray*} \label{tbtm}
\vline \frac{T_{2,n}(t_n)}{T_n(t_n)} - \frac{\hat{T}_{2,n}(t_n)}{\hat{T}_n(t_n)} \vline \, \, \leq \, \, \frac{1+n^2}{n^4}.
\end{eqnarray*}
Combining the last two inequalities, we get \eqref{tvtm}.

\vspace{0.2 cm}
\noindent {\bf  II. Proof of \eqref{thtm}}. 

\noindent {\bf  IIa. Estimate of the quotient $x_j/x_{j_*}$}.  We first observe that when $|j-j_*| < n^{5/6}$,
\begin{equation} \label{mjj}
\sqrt{ \frac{j_*(n-j_*)}{j(n-j)}}  = 1+ \kO(|j-j_*|/n)=1+\kO(n^{-1/6}).
  \end{equation}
It follows from   Lemma \ref{lgb} (i) that for all $j$,
  \begin{equation} \label{mgjj}
  \Big| \big[\log g(\beta,dj,dn)-ndF(j/n)\big]  - \big[\log g(\beta,dj_*,dn)-ndF(j_*/n)\big]\Big| =\kO(|j-j_*|/n).
  \end{equation}
As for \eqref{nljj}, by using \eqref{lpjn} we have for all $|j-j_*| < n^{5/6}$,
  \begin{eqnarray*}
  n(L(j/n)-L(j_*/n)) &=& n \left(L'\left(\frac{j_*}{n}\right) \left(\frac{j-j_*}{n}\right) + L''\left(\frac{j_*}{n}\right) \frac{(j-j_*)^2}{2n^2} + \kO\left( \left(\frac{j-j_*}{n}  \right)^3\right) \right) \\
 & = & L''\left(\frac{j_*}{n} \right) \frac{(j-j_*)^2}{2n} + \kO\left(\frac{j-j_*}{n}  \right) + \kO\left( \frac{(j-j_*)^3}{n^2}  \right)\\
  & = & \left[ L''\left(\frac{j_*}{n} \right) + \kO(n^{-1/6})\right] \frac{(j-j_*)^2}{2n} + \kO(n^{-1/6}),
  \end{eqnarray*}
where in the  last line,  we used that 
  $$\kO\left( \frac{(j-j_*)^3}{n^2}  \right) =  \frac{(j-j_*)^2}{2n} \kO\left( \frac{(j-j_*)}{n}  \right)= \frac{(j-j_*)^2}{2n}\kO\left( n^{-1/6} \right).$$
  On the other hand,
  \begin{equation*}
  L''(j_*/n) =H''(j_*/n) =H''(t_*) + \kO(1/n).
  \end{equation*}
  Therefore
    \begin{equation} \label{mljj}
  n \big(L(j/n)-L(j_*/n) \big)= \left[\frac{H''(t_*)}{2}+ \kO(n^{-1/6}) \right] \frac{(j-j_*)^2}{n} + \kO(n^{-1/6}). 
  \end{equation} 
  Let us define 
  \[\alpha_*:=H''(t_*)/2 = L''(t_*)/2 <0.\]  
  Using \eqref{tjj}, \eqref{mjj}, \eqref{mgjj} and \eqref{mljj}, we get that for any $\varepsilon \in (0, |\alpha_*|/8)$, for all $n$ large enough and  $|j-j_*| < n^{5/6}$,
  \begin{equation} \label{uxj}
 \frac{x_j}{ x_{j_*}} \leq (1+\varepsilon)\exp\left( \left(\alpha_* + \varepsilon\right)\frac{(j-j_*)^2}{n} + 2t_n(j-j_*)\right) ,
  \end{equation}
  and 
   \begin{equation} \label{lxj}
  \frac{x_j}{ x_{j_*}} \geq (1-\varepsilon)\exp\left( \left(\alpha_* - \varepsilon\right)\frac{(j-j_*)^2}{n} + 2t_n(j-j_*)\right).
  \end{equation}
{\bf  IIb. Estimate of $\hat{T}_n(t_n)$}.   Observe that
\begin{eqnarray} \label{tnbs}
\hat{T}_n(t_n) = \sum_{|j-j_*| < n^{5/6}} x_j = x_{j_*} \sum_{|j-j_*| < n^{5/6}} \frac{x_j}{x_{j_*}}.
\end{eqnarray}  
Moreover, for all $\alpha < 0$,
\begin{eqnarray}
\sum_{|j-j_*| < n^{5/6}}  \exp\left(  \frac{\alpha(j-j_*)^2}{n} + 2t_n(j-j_*)\right) &= & \sum_{|j| < n^{5/6}} \exp\left( \alpha \frac{j^2}{n} + 2(t_n \sqrt{n}) \frac{j}{\sqrt{n}}\right) \notag \\
&= & \sum_{j =- \infty }^{\infty } \exp\left( \alpha \frac{j^2}{n} + 2(t_n \sqrt{n}) \frac{j}{\sqrt{n}}\right) +o(1) \notag \\
&=& \sqrt{n} \int_{-\infty}^{\infty} e^{\alpha x^2 +2(t_n\sqrt{n})x} dx +o(1). \label{sctn}
\end{eqnarray}
  Here we used the integral approximation and the fact that $t_n \sqrt{n}$ is uniformly bounded.
Combining \eqref{uxj}, \eqref{lxj},  \eqref{sctn} yields that 
\begin{eqnarray} \label{tjjs}
(1-2\varepsilon) \sqrt{n}  A(\alpha_*-\varepsilon, t_n \sqrt{n}) \leq \sum_{|j-j_*| < n^{5/6}} \frac{x_j}{x_{j_*}}  \leq (1+2\varepsilon) \sqrt{n}  A(\alpha_*+\varepsilon, t_n \sqrt{n}),
\end{eqnarray}
where for  $\alpha <0$ and $\gamma \in \mathbb{R}$,
  \begin{equation*}
  A(\alpha, \gamma):=\int_{-\infty}^{\infty} e^{\alpha x^2+ \gamma x} dx.
  \end{equation*}
Using \eqref{tnbs} and \eqref{tjjs}, we have
  \begin{eqnarray} \label{tktn}
(1-2\varepsilon) \sqrt{n} x_{j_*}  A(\alpha_*-\varepsilon, t_n \sqrt{n}) \leq \hat{T}_n(t_n) \leq (1+2\varepsilon) \sqrt{n} x_{j_*} A(\alpha_*+\varepsilon, t_n \sqrt{n}).
\end{eqnarray}
{\bf  IIc. Estimate of $\hat{T}_{1,n}(t_n)$}. 
We have 
  \begin{eqnarray} \label{tmtn}
\hat{T}_{1,n}(t_n)& =& \sum_{|j-j_*| < n^{5/6}} jx_j = x_{j_*}  \sum_{|j-j_*| < n^{5/6}} (j-j_*)\frac{x_j}{x_{j_*}} + j_* x_{j_*} \sum_{|j-j_*| < n^{5/6}} \frac{x_j}{x_{j_*}}. 
\end{eqnarray}  
As for \eqref{sctn}, we have  
  \begin{eqnarray*} \label{jxj}
\hspace{-0.9 cm} \sum_{|j-j_*| < n^{5/6}} (j-j_*)  \exp\left(  \frac{\alpha(j-j_*)^2}{n} + 2t_n(j-j_*)\right) &=& n \int_{-\infty}^{\infty} x e^{\alpha x^2 +2(t_n\sqrt{n})x} dx +o(1).
\end{eqnarray*}  
 Using this  approximation and \eqref{uxj}, \eqref{lxj}, we get  
  \begin{eqnarray} \label{xjjs}
\hspace{-0.9 cm} (1-2\varepsilon) n  A_1(\alpha_*-\varepsilon, t_n \sqrt{n}) \leq \sum_{|j-j_*| < n^{5/6}} (j-j_*)\frac{x_j}{x_{j_*}}  \leq (1+2\varepsilon) n  A_1(\alpha_*+\varepsilon, t_n \sqrt{n}),
\end{eqnarray}
where for  $\alpha <0$ and $\gamma \in \mathbb{R}$,
  \begin{equation*}
  A_1(\alpha, \gamma):=\int_{-\infty}^{\infty} x e^{\alpha x^2 +\gamma x} dx.
  \end{equation*}
Combining \eqref{tjjs}, \eqref{tmtn}, \eqref{xjjs}, we get 
  \begin{eqnarray} \label{nhsa}
(1-2\varepsilon)  x_{j_*} [ \sqrt{n} j_* A(\alpha_*-\varepsilon, t_n \sqrt{n}) + n A_1(\alpha_*-\varepsilon, t_n \sqrt{n})] \leq \hat{T}_{1,n}(t_n) \\
 \vspace{9 cm}\leq (1+2\varepsilon)  x_{j_*} [ \sqrt{n} j_* A(\alpha_*+\varepsilon, t_n \sqrt{n})+ n A_1(\alpha_*+\varepsilon, t_n \sqrt{n})] \notag
\end{eqnarray}
 {\bf IId. Estimate of $\hat{T}_{2,n}(t_n)$}. We observe that
    \begin{eqnarray} \label{thtn}
\hat{T}_{2,n}(t_n) = \sum_{|j-j_*| < n^{5/6}} j^2x_j &=& x_{j_*}  \sum_{|j-j_*| < n^{5/6}} (j-j_*)^2\frac{x_j}{x_{j_*}}+ 2j_* x_{j_*} \sum_{|j-j_*| < n^{5/6}} (j-j_*)\frac{x_j}{x_{j_*}}  \notag \\
&+ &   j_*^2 x_{j_*} \sum_{|j-j_*| < n^{5/6}} \frac{x_j}{x_{j_*}}. 
\end{eqnarray}  
Similar to  \eqref{sctn}, 
\begin{eqnarray*}
\sum_{|j-j_*| < n^{5/6}} (j-j_*)^2  \exp\left(  \frac{\alpha(j-j_*)^2}{n} + 2t_n(j-j_*)\right) &=& n\sqrt{n} \int_{-\infty}^{\infty} x^2 e^{\alpha x^2 +2(t_n\sqrt{n})x} dx +o(1).
\end{eqnarray*}
 Using this equality and \eqref{uxj}, \eqref{lxj}, we obtain 
  \begin{eqnarray*} \label{ah}
\vspace{-0.4 cm} (1-2\varepsilon) n \sqrt{n} A_2(\alpha_*-\varepsilon, t_n \sqrt{n}) \leq \sum_{|j-j_*| < n^{5/6}} (j-j_*)^2\frac{x_j}{x_{j_*}}  \leq (1+2\varepsilon) n\sqrt{n}  A_2(\alpha_*+\varepsilon, t_n \sqrt{n}),
\end{eqnarray*}
where for  $\alpha <0$ and $\gamma \in \mathbb{R}$,
  \begin{equation*}
  A_2(\alpha, \gamma):=\int_{-\infty}^{\infty} x^2 e^{\alpha x^2 +\gamma x} dx.
  \end{equation*}
  Combining this estimate with \eqref{tmtn}, \eqref{xjjs} and \eqref{thtn}, we have 
    \begin{eqnarray} \label{nhtm}
 && \vspace{-0.4 cm} (1-2\varepsilon)   x_{j_*} \big [ j_*^2 \sqrt{n} A (\alpha_* - \varepsilon, 2 t_n \sqrt{n}) + 2j_*n A_1(\alpha_* - \varepsilon, 2 t_n \sqrt{n}) + n\sqrt{n} A_2(\alpha_*- \varepsilon, 2 t_n \sqrt{n}) \big ] \leq \hat{T}_{2,n}(t_n) \notag \\
&& \leq (1+2\varepsilon) x_{j_*} \big [ j_*^2 \sqrt{n} A (\alpha_* + \varepsilon, 2 t_n \sqrt{n}) + 2j_*n A_1(\alpha_* + \varepsilon, 2 t_n \sqrt{n}) + n\sqrt{n} A_2(\alpha_* + \varepsilon, 2 t_n \sqrt{n}) \big ]. 
\end{eqnarray}
{\bf  IIe. Conclusion}.  We observe that the derivatives with respect to $\alpha$ at $\alpha_*$  of the functions $A(\alpha, \gamma), A_1(\alpha, \gamma)$ and $A_2(\alpha, \gamma)$ are bounded. Hence, for any $t>0$, there exists a positive constant $C=C(t)$, such that for all $|\gamma| \leq t$, 
   \begin{eqnarray}
  |A(\alpha_* \pm \varepsilon, \gamma) - A(\alpha_*, \gamma)| & \leq & C A(\alpha_*, \gamma) \varepsilon, \label{ake} \\
    |A_1(\alpha_* \pm \varepsilon, \gamma) - A_1(\alpha_*, \gamma)|& \leq & C A_1(\alpha_*, \gamma) \varepsilon, \label{ame}\\
      |A_2(\alpha_* \pm \varepsilon, \gamma) - A_2(\alpha_*, \gamma)| &\leq& C A_2(\alpha_*, \gamma) \varepsilon. \label{ahe}
   \end{eqnarray}
Using  \eqref{tktn} and  \eqref{ake} we get that for any $\varepsilon \in (0, |\alpha_*|/8)$ and  $n$ large enough 
\begin{eqnarray} \label{btn}
\lambda^-_{\varepsilon} B_n \leq \hat{T}_{n}(t_n)  \leq \lambda^+_{\varepsilon} B_n,
\end{eqnarray}
where 
\[B_n = x_{j_*} \sqrt{n} A (\alpha_*, 2 t_n \sqrt{n}), \]
and 
$$ \lambda^-_{\varepsilon}=(1 - C \varepsilon)(1-2\varepsilon), \hspace{3cm} \lambda^+_{\varepsilon}=(1 + C \varepsilon)(1+2\varepsilon). $$
Similarly, by using \eqref{nhsa} and \eqref{ame} we get 
\begin{eqnarray} \label{btmn}
\lambda^-_{\varepsilon} B_{1,n} \leq \hat{T}_{1,n}(t_n)  \leq \lambda^+_{\varepsilon} B_{1,n},
\end{eqnarray}
where 
\[B_{1,n} = x_{j_*} \big [ j_* \sqrt{n} A (\alpha_*, 2 t_n \sqrt{n}) + n A_1(\alpha_*, 2 t_n \sqrt{n}) \big ],\]
and by using \eqref{nhtm} and \eqref{ahe},
\begin{eqnarray} \label{bthn}
\lambda^-_{\varepsilon} B_{2,n} \leq \hat{T}_{2,n}(t_n)  \leq \lambda^+_{\varepsilon}  B_{2,n},
\end{eqnarray}
where 
\[B_{2,n} = x_{j_*} \big [ j_*^2 \sqrt{n} A (\alpha_*, 2 t_n \sqrt{n}) + 2j_*n A_1(\alpha_*, 2 t_n \sqrt{n}) + n\sqrt{n} A_2(\alpha_*, 2 t_n \sqrt{n}) \big ].\]
Combining \eqref{btn}, \eqref{btmn}, \eqref{bthn}, we have
\begin{eqnarray} \label{hhtm}
\lambda^-_{\varepsilon} \left[ \frac{B_{2,n}}{B_{n}} - \left( \frac{B_{1,n}}{B_n} \right)^2\right] \leq  \frac{\hat{T}_{2,n}(t_n)}{\hat{T}_n(t_n)} - \left( \frac{\hat{T}_{1,n}(t_n)}{\hat{T}_n(t_n)}\right)^2 \leq  \lambda^+_{\varepsilon}  \left[ \frac{B_{2,n}}{B_{n}} - \left( \frac{B_{1,n}}{B_n} \right)^2\right]. 
\end{eqnarray}
Moreover, 
\begin{eqnarray*}
\frac{B_{2,n}}{B_{n}} - \left( \frac{B_{1,n}}{B_n} \right)^2 = n \left[ \frac{A_{2}(\alpha_*, 2 t_n \sqrt{n})}{A(\alpha_*, 2 t_n \sqrt{n})} - \left( \frac{A_{1}(\alpha_*, 2 t_n \sqrt{n})}{A(\alpha_*, 2 t_n \sqrt{n})} \right)^2\right].
\end{eqnarray*}
Note that  $A(\alpha, \gamma), A_1(\alpha, \gamma), A_2(\alpha, \gamma)$ are related to moments of the normal distribution with mean $\gamma/(2\alpha)$ and variance $1/(-2 \alpha)$. By some simple calculus, we have 
\begin{equation*} \label{ahm}
\frac{A_2(\alpha, \gamma)}{A(\alpha, \gamma)} - \left( \frac{A_1(\alpha, \gamma)}{A(\alpha, \gamma)} \right)^2 = \frac{-1}{2 \alpha}.
\end{equation*}
Thus 
\begin{eqnarray} \label{bhmn}
\frac{B_{2,n}}{B_{n}} - \left( \frac{B_{1,n}}{B_n} \right)^2  = \frac{-n}{2 \alpha_*} = \frac{-n}{H''(t_*)}.
\end{eqnarray}
Combining   \eqref{hhtm} and \eqref{bhmn} yields that 
\[   \lambda^-_{\varepsilon}\frac{-4}{H''(t_*)} \leq   \frac{4}{n} \left(\frac{\hat{T}_{2,n}(t_n)}{\hat{T}_n(t_n)} - \left( \frac{\hat{T}_{1,n}(t_n)}{\hat{T}_n(t_n)}\right)^2 \right)   \leq  \lambda^+_{\varepsilon}   \frac{-4}{H''(t_*)}.\]
Letting $\varepsilon \rightarrow 0$ and $n\rightarrow \infty$, we get \eqref{thtm}. \hfill $\square$
\section{Proof of Proposition \ref{pnl} }
In this section, we assume that $\beta > \beta_c$ and $B=0$. Then for all $\sigma \in \Omega_n$, 
\begin{equation} \label{msts}
\mu_n(\sigma) = \mu_n(-\sigma).
\end{equation}
By this symmetry of the measure $\mu_n$, we observe that   Proposition \ref{pnl} follows if there exists a positive constant $\nu$, such that   as $n \rightarrow \infty$
 \begin{equation} \label{dm}
 \mu_n \left( \, \vline \, \frac{S_n}{n} - \nu \,  \vline \, < n^{-1/6}  \right)  \longrightarrow 1/2.
 \end{equation}
We now  prove \eqref{dm} using  the same strategy as in Section 5. By \eqref{ehs} and \eqref{ezn}, we have  
\begin{equation} \label{musg}
\mu_n(\sigma) = \frac{g(\beta, d |\sigma_+|, dn)}{\sum \limits_{j=0}^n \binom{n}{j}g(\beta, dj, dn)}.
\end{equation}
We have proved in the Claim 2a in Section 4 that on $(1/2, 1)$, the function $H'(t)$ has a unique zero $t_+$, which is  the maximum point of $H(t)$. Let us define 
$$\nu= 2t_+ - 1.$$ 
Since $S_n = 2 |\sigma_+| -n$,
 \begin{equation*} 
 \mu_n \left( \, \vline \, \frac{S_n}{n} - \nu \,  \vline \, < n^{-1/6}  \right) =  \mu_n \left( \, \vline \, \frac{|\sigma_+|}{n} - t_+ \,  \vline \, < n^{-1/6}  \right). 
 \end{equation*}
 Combining this with \eqref{musg}, we get 
  \begin{equation} \label{nubb} 
 \mu_n \left( \, \vline \, \frac{S_n}{n} - \nu \,  \vline \, < n^{-1/6}  \right) = \frac{\sum \limits_{j=0}^n y_j 1(|(j/n)-t_+| < n^{-1/6})}{\sum \limits_{j=0}^n y_j},
 \end{equation}
 where 
 $$y_j=\binom{n}{j}g(\beta, dj, dn).$$
We note that $y_j = y_{n-j}$, so 
$$\sum \limits_{j=0}^n y_j= 2 \left( \sum \limits_{j=[n/2]+1}^n y_j + z/2 \right) =: 2 R_n,$$
where 
$$
z = \begin{cases}
y_{[n/2]} \quad \textrm{ if $n$ is odd } \\
0 \quad \qquad \textrm{ otherwise}.
\end{cases}
$$
 Let us define
$$\hat{R}_n = \sum \limits_{j=0}^n y_j 1(|(j/n)-t_+| < n^{-1/6}) \hspace{0.9 cm} \textrm{and} \hspace{0.9 cm} \bar{R}_n = R_n - \hat{R}_n.$$ 
 Note that $j_+ \in (1/2, 1)$, so all the indies in the definition of $\hat{R}_n$ are in the sum $R_n$. Now we observe that by \eqref{nubb}, the equation \eqref{dm}  is equivalent to 
 \begin{equation} \label{rhrn}
 \frac{\hat{R}_n}{R_n} \rightarrow 1.
 \end{equation}
Using the same idea of the proof of Lemma \ref{lcgf}, we define
 $$j_+=[nt_+].$$
As for \eqref{tjj}, for all $j$
\begin{eqnarray*}
\frac{x_j}{x_{j_+}}=(1+o(1)) \sqrt{\frac{j_+(n-j_+)}{j(n-j)}} & \exp \Big(   n\left[H(j/n)-H(j_+/n)\right] 
+ \left[\log g(\beta,dj,dn)-ndF(j/n) \right] \\
& \hspace{0.9 cm} - \left[\log g(\beta,dj_+,dn)-ndF(j_+/n) \right] \Big).
\end{eqnarray*} 
 Using the same arguments for \eqref{xjsj}, we can show that for all $j$ satisfying  $j\geq [n/2]$ and $|(j/n)-t_+| \geq n^{-1/6}$, 
 \begin{equation} \label{xjjp}
 x_j \leq \frac{x_{j_+}}{n^7}.
 \end{equation}
 Note that here $H(t)$ and $t_+$ play the same role of $L(t)$ and $t_*$   as in the proof of \eqref{xjsj}. Using \eqref{xjjp}, we get 
 \begin{equation*}
 \bar{R}_n \leq \frac{x_{j_+}}{n^6} \leq \frac{\hat{R}_n}{n^6}.  
 \end{equation*}
 Thus 
 \begin{equation*}
 \frac{\hat{R}_n}{R_n}= \frac{\hat{R}_n}{\bar{R}_n+ \hat{R_n}} \longrightarrow 1,
 \end{equation*}
 and  \eqref{rhrn} follows. \hfill $\square$
\section{The annealed pressure of  Ising model on the configuration model}
Let $G_n$ be the configuration model whose the vertex set is $V_n=\{v_1,\ldots,v_n\}$ and the  degrees of vertices $(D_i)$ are  i.i.d.  integer-valued random variables with the same distribution as $D$. Assume that   
\begin{equation} \label{cdd}
\E\left( e^{s D}\right) < \infty \quad \textrm{for all} \quad s \in \R.
\end{equation}
Notice that the condition \eqref{cdd} is necessary, since  without it the partition function  has infinite expectation  when  $\beta$ is large enough. 

Now  we  study the annealed pressure of the Ising model on $G_n$.   We use the same notation as in the proof of Theorem \ref{ttd} (i). Observe that for all $\sigma \in \Omega_n$, 
\begin{eqnarray*}
\sum_{i=1}^n \sigma_i &= &2|\sigma_+| -n, \\
\sum_{i\leq j} k_{i,j} \sigma_i \sigma_j &=& \ell_n/2 - 2e(\sigma_{+}, \sigma_{-}),
\end{eqnarray*}
where for all $1\leq j \leq n$,
\[\ell_j=D_1+\ldots+D_j.\]
Using \eqref{hebn} and the fact that $(D_i)_{1 \leq i \leq n}$ are i.i.d. random variables, we have if $|\sigma_+| = |\sigma_+'|$,
\[\ell_n/2 - 2e(\sigma_{+}, \sigma_{-}) \mathop{=}^{(\kD)} \ell_n/2 - 2e(\sigma'_{+}, \sigma'_{-}).\]
Hence using the same arguments as for Theorem \ref{ttd} (i), we obtain
\begin{equation} \label{zbb}
\E(Z_n(\beta,B))= e^{-Bn}\sum_{j=0}^n \binom{n}{j} e^{2Bj} b(\beta,j,n),
\end{equation}
where 
\[b(\beta,j,n)= \E \Big(\exp \big[\beta \ell_n/2 -2 \beta e(U_j,U_j^c)\big]\Big),\]
with \[U_j=\{v_1,\ldots,v_j\}.\]
Using  \eqref{hebn} once again,  we have
\[ \kL(e(U_j,U_j^c)\mid (D_i)_{1\leq i \leq n} ) \mathop{=}^{(\kD)} \kL(X(\ell_j,\ell_n) \mid (D_i)_{1\leq i \leq n} ), \]
where  $X(k,m)$ is defined as in \eqref{dfox} for all $k\leq m$. Hence
 \[\E_{(D_i)} \Big(\exp \big[ -2 \beta e(U_j,U_j^c) \big]   \Big)=g(\beta,\ell_j,\ell_n),\]
 where $\E_{(D_i)}$ is the expectation w.r.t. configuration model conditioning on the sequence of degrees $(D_i)_{i\leq n}$, and  $g(\beta,k,m)$ is defined as in \eqref{dfog}.  Thus
\[\E_{(D_i)}\Big(\exp \big[\beta \ell_n/2 -2 \beta e(U_j,U_j^c) \big ]  \Big)=\exp(\beta\ell_n/2)g(\beta,\ell_j,\ell_n).\]
Therefore
\[b(\beta,j,n)=\bar{\E} \left( \E_{(D_i)}\Big(\exp \big[\beta \ell_n/2 -2 \beta e(U_j,U_j^c) \big ] \Big) \right)= \bar{\E} \Big(\exp(\beta\ell_n/2)g(\beta,\ell_j,\ell_n)\Big),\]
where $\bar{\E}$ is the expectation w.r.t. the sequence of degrees $(D_i)_{i\leq n}$.
By Lemma \ref{lgb} (ii), there is a positive constant $C=C(\beta)$, such that for all $j\leq n$
\[\exp \Big(-C+ \ell_nF(\ell_j/\ell_n)\Big) \leq g(\beta, \ell_j, \ell_n) \leq \exp \Big(C+ \ell_nF(\ell_j/\ell_n)\Big),\]
with  $F(t)$ as in Theorem \ref{ttd} (i). Hence
\begin{equation} \label{bbjn}
 \vline \, \frac{\log b(\beta,j,n)}{n} - \frac{\log \bar{\E} \Big(\exp \big[\ell_n (\beta/2+F(\ell_j/\ell_n))\big]\Big)}{n}  \, \vline \leq \frac{C}{n}.
\end{equation}
For each $\beta \geq 0$, we define a sequence of functions on $[0,1]$ as follows:
\[G_n(\beta,t)= \frac{1}{n} \log \bar{\E}\Big(\exp \big[\ell_n (\beta/2+F(\ell_{[nt]}/\ell_n))\big]\Big).\]
To study the limit of the sequence of functions $(G_n(\beta,t))_n$, we need a  large deviation result for the vector $(\ell_{[nt]},\ell_n)$. We use the standard notion of large deviation principle (LDP) as in \cite{DZ}.
\iffalse
 We first recall a well-known result of Varadhan.
\begin{lem}
Let $(\mu_n)_n$ be a sequence of measures on $\mathcal{X}$ satisfying the LDP with a rate function $R: \mathcal{X} \rightarrow [0, \infty]$. Let $\varphi$ be any  continuous function on $\kX$ satisfying 
$$\limsup \limits_{n \rightarrow \infty} \frac{1}{n} \log \int e^{\gamma n \varphi (x)} d\mu_n(x) < \infty, $$
for some constant $\gamma >1$. Then
 $$\lim \limits_{n \rightarrow \infty} \frac{1}{n} \log \int e^{ n \varphi (x)} d\mu_n(x) \rightarrow \sup \limits_{x \in \kX} \{ \varphi(x) - R(x)\}.$$
 \end{lem}
 \fi
 Let $(X_i)$ be a sequence of i.i.d. random variables. Suppose that  for all $s \in \R$
$$\Lambda(s)=\E(e^{s X_1}) < \infty.$$
Let us define for $t\in [0, 1]$,
$$Z_n(t)= \frac{1}{n} \sum \limits_{i=1}^{[nt]} X_i.$$
Let $\nu_n$ be the law of $Z_n(\cdot)$ in $L_{\infty}([0,1])$. 
\begin{lem} \label{ldp} \cite[Lemma 5.1.8]{DZ}
Let $\kQ$ denote the collection of all ordered finite subsets of $(0,1]$. For any $q= \{0< t_1< \ldots<t_{|q|} \leq 1\} \in \kQ$ and $f: [0, 1] \rightarrow \R$, let $p_q(f)$ denote the vetor $(f(t_1), \ldots, f(t_{|q|}))\in \R^{|q|}$. Then the sequence of laws $(\nu_n \circ p_q^{-1})_n$ satisfies the LDP in $\R^{|q|}$ with the rate function 
$$R_q(z)=\sum_{i=1}^{|q|} (t_i - t_{i-1}) \Lambda^* \left( \frac{z_i - z_{i-1}}{t_i - t_{i-1}} \right),$$
where $z=(z_1, \ldots, z_{|q|})$, $z_0=t_0=0$, and 
$$\Lambda^*(x)=\sup_{s \in \mathbb{R}}\{ x s- \Lambda(s)\}.$$  
\end{lem}   
\noindent Using this result, we can  show the convergence of the sequence $(G_n(\beta,t))_n$.
\begin{lem} \label{lgc}
For all $\beta \geq 0$, the following assertions hold.
\begin{itemize} 
\item[(i)] There exists a positive constant $C$, such that for all $0 \leq s, t \leq 1$ and $n \geq 1$,
\[|G_n(\beta,t)-G_n(\beta,s)| \leq C \left(|t-s|+ \frac{1}{n} \right).\]
\item[(ii)] For all $t \in [0,1]$, we have
\[ \lim_{n\rightarrow \infty} G_n(\beta,t) =G(\beta,t),\]
where
\[G(\beta,t)=\sup_{a,b} \Big\{b(\beta/2 +F(a/b)) - t \Lambda^*\left( \frac{a}{t}\right) - (1-t) \Lambda^*\left( \frac{b-a}{1-t}\right) \Big\},\]
with 
\[\Lambda^*(x)=\sup_{s \in \mathbb{R}}\{ x s- \Lambda(s)\},\]
and 
\[\Lambda (s)= \log \E(\exp(s D)).\]
Moreover, $G(\beta,t)$ is   a Lipschitz function.
\end{itemize}
\end{lem}
\begin{proof}  We first prove (i). Observe that 
\begin{equation*}
0 \leq F(t) \leq 1 \quad \textrm{and} \quad \max_{t \in [0,1]} |F'(t)| = \max_{t \in [0,1/2]} |\log f(t)| \leq 2 \beta, 
\end{equation*} 
since the function $F(t)$ is symmetric about $1/2$ and  $e^{-2 \beta} \leq f(\beta,t) \leq 1$ for all $t \in [0, 1/2]$. Therefore
\begin{equation} \label{fcfp}
\beta/2 + \max_{t \in [0,1]} (|F(t)| + |F'(t)|) \leq r:= 1 + 5 \beta/2.
\end{equation}
We claim that 
\begin{equation} \label{tata}
|\log \bar{\E} \left(e^{\ell_n[\beta/2 + F(\ell_{j}/ \ell_n)]} \right) - \log \bar{\E} \left(e^{\ell_n[\beta/2 + F(\ell_{j-1}/ \ell_n)]} \right)| \leq C,
\end{equation} 
 where
 \begin{equation*}
 C= \max \{ \log \bar{\E} \left( e^{3rD} \right)   -\log \bar{\E} \left( e^{-2rD} \right), \log \bar{\E}\left( e^{2rD} \right) - \log \bar{\E}\left( e^{-3rD} \right) \}. 
 \end{equation*}
Assuming \eqref{tata}, we can easily prove (i). Indeed, by repeatedly applying \eqref{tata}, we have for all $i \leq j \leq n$,
 \begin{equation}
\Big|\log \bar{\E} \left(e^{\ell_n[\beta/2 + F(\ell_{i}/ \ell_n)]} \right) - \log \bar{\E} \left(e^{\ell_n[\beta/2 + F(\ell_{j}/ \ell_n)]} \right) \Big| \leq C|i-j|.
\end{equation} 
 Thus 
 \begin{eqnarray*}
 |G_n(\beta,t)-G_n(\beta,s)| &=& \frac{1}{n} \Big|\log \bar{\E} \left(e^{\ell_n[\beta/2 + F(\ell_{[nt]}/ \ell_n)]} \right) - \log \bar{\E} \left(e^{\ell_n [\beta/2 + F(\ell_{[ns]}/ \ell_n)]} \right) \Big| \\
 &\leq & C \left(|t-s|+ \frac{1}{n} \right),
 \end{eqnarray*}
 which implies (i).

\noindent{\it Proof of \eqref{tata}}. The idea is simple: using  the mean value theorem and \eqref{fcfp}, we have for all $1 \leq j \leq n$,
 \begin{equation} \label{fent}
 \ell_n |F(\ell_{j}/ \ell_n) - F(\ell_{j-1}/ \ell_n)| \leq  \max_{t \in [0,1]} |F'(t)| D_j\leq r D_j.
 \end{equation}
 Hence \eqref{tata} would immediately follow if $\ell_n$ and $D_i$ are independent. Since this fact is not true, we break $\ell_n$ into two independent parts $D_i$ and $\ell_{n,i}$, with 
  \[\ell_{n,j}= \ell_n - D_j.\]
We have
 \begin{eqnarray*}
& &\ell_n(\beta/2 + F(\ell_{j-1}/ \ell_n)) =  \ell_{n,j} (\beta/2 + F(\ell_{j-1}/ \ell_n)) + D_j ( \beta/2 +F(\ell_{j-1}/ \ell_n)) \\
&=&  \ell_{n,j} (\beta/2 + F(\ell_{j-1}/ \ell_{n,j})) + \ell_{n,j} (F(\ell_{j-1}/ \ell_{n}) - F(\ell_{j-1}/ \ell_{n,j})) + D_j( \beta/2 + F(\ell_{j-1}/ \ell_n)).
 \end{eqnarray*}
 Using the mean value theorem and \eqref{fcfp}, we get 
 \begin{equation*}
\Big |\ell_{n,j} \big(F(\ell_{j-1}/ \ell_{n}) - F(\ell_{j-1}/ \ell_{n,j}) \big) + D_j \big(\beta/2 + F(\ell_{j-1}/ \ell_n)\big) \Big| \leq 2 r D_j.
 \end{equation*}
 Therefore
 \begin{equation} \label{ljtm}
 \Big|\ell_n \big(\beta/2 + F(\ell_{j-1}/ \ell_n) \big)- \ell_{n,j} \big(\beta/2 + F(\ell_{j-1}/ \ell_{n,j})\big) \Big| \leq 2rD_j.
 \end{equation}
It follows from \eqref{fent} and \eqref{ljtm} that
 \begin{equation} \label{bkdj}
 \Big|\ell_n \big(\beta/2 + F(\ell_{j}/ \ell_n) \big)- \ell_{n,j} \big(\beta/2 + F(\ell_{j-1}/ \ell_{n,j})\big) \Big| \leq 3rD_j.
 \end{equation}
On the other hand,
\begin{equation} \label{ljdl}
\ell_{n,j} \big(\beta/2 + F(\ell_{j-1}/ \ell_{n,j})\big) \qquad \textrm{is independent of} \qquad D_j.
\end{equation} 
 Using \eqref{bkdj} and \eqref{ljdl}, we obtain
\begin{equation*}
  \bar{\E} \left( e^{-3r D_j}\right) \leq \frac{ \bar{\E}  \big(e^{\ell_n [\beta/2 + F(\ell_{j}/ \ell_n)]} \big)}{\bar{\E} \big(e^{\ell_{n,j} [\beta/2 + F(\ell_{j-1}/ \ell_{n,j})]} \big )} \leq \bar{\E} \left( e^{3r D_j}\right).
\end{equation*} 
 Similarly, using \eqref{ljtm} and \eqref{ljdl}, we have  
 \begin{equation*}
  \bar{\E} \left( e^{-2r D_j}\right) \leq \frac{ \bar{\E} \left(e^{\ell_n[\beta/2 + F(\ell_{j-1}/ \ell_n)]} \right)}{\bar{\E} \left(e^{\ell_{n,j} [\beta/2 + F(\ell_{j-1}/ \ell_{n,j})]} \right)} \leq \bar{\E} \left( e^{2r D_j}\right).
\end{equation*} 
Combining the last two inequalities gives \eqref{tata}. 

 We now prove (ii).  Applying Lemma \ref{ldp} for $q=\{t_1=t < t_2= 1\}$, we get that the law of $ \frac{1}{n}(\ell_{[nt]}, \ell_n)$ satisfies the LDP  in $\R^2$ with the rate function
\[I(a,b)=t \Lambda^*\left( \frac{a}{t}\right) + (1-t) \Lambda^*\left( \frac{b-a}{1-t}\right),\]
where  $\Lambda^*$ is defined as in the statement of (ii). Therefore, using Varadhan's Lemma (see for example \cite[Theorem 4.3.1]{DZ}), the sequence of functions $(G_n(\beta, \cdot))_n$ converges point-wise to the function $G(\beta,\cdot)$ defined as in the statement of (ii). Moreover, applying Lemma \ref{lpre} (ii) and Part i, we obtain that $G(\beta,t)$ is a Lipschitz function.
\end{proof}
\begin{prop} \label{pbbc} For all $\beta \geq 0$ and $B \in \R$, the annealed pressure $\psi_n(\beta,B)$ converges to a limit given by 
\[\psi(\beta,B)= -B + \max_{0\leq t \leq 1} \left[t\log \left( \frac{1}{t}\right)+(1-t)\log \left( \frac{1}{1-t}\right)+2Bt+ G(\beta, t)\right],\]
with $G(\beta,t)$ as in Lemma \ref{lgc}. 
\end{prop}
\begin{proof}
Using  \eqref{zbb}, \eqref{bbjn} and Stirling's formula, we get 
\begin{eqnarray} \label{lzbb}
\frac{\log \E(Z_n(\beta,B))}{n} &=& -B+ \max_{0 \leq j \leq n} \left[ \frac{ \log \binom{n}{j}}{n} + 2B \frac{j}{n} + G_n(\beta, j/n) \right] + o(1) \notag \\
&=& -B  + \max_{0 \leq j \leq n} \left[ S(j/n) + G_n(\beta, j/n) \right] + o(1),
\end{eqnarray}
where  $S(t)$ is continuous function on $[0,1]$ defined by
\[S(t)= -t \log t + (t-1) \log (1-t) + 2Bt.\]
Now it follows from \eqref{lzbb},  Lemmas \ref{lgc} (ii)  and \ref{lpre} (ii)  that 
\[\lim_{n \rightarrow \infty} \frac{\log \E(Z_n(\beta,B))}{n} =  -B + \max_{0\leq t \leq 1} \left[S(t)+ G(\beta, t)\right],\]
which proves Proposition \ref{pbbc}.
\end{proof}
\begin{rem} \emph{ We can slightly extend Proposition \ref{pbbc} as follows.  Let  $(X_n)_{n\geq 1}$   and   $X$ be integer valued random variables satisfying 
 \begin{eqnarray}
 \sup_{s \in \R} \E(e^{sX}) < \infty \hspace{1cm} \textrm{ and } \hspace{1cm}   \E(e^{sX_n}) \longrightarrow  \E(e^{sX}) \, \forall \, s \in \R.   
 \end{eqnarray}
For each $n$, let $(X_{n,i})_{i\leq n}$ be a sequence of i.i.d random variables with the same distribution as $X_n$. Let $G_n$ be the configuration model random graph of size $n$ with sequence of degrees given by $(X_{n,i})_{i\leq n}$. Then   Proposition \ref{pbbc} still holds for the annealed Ising model on $G_n$.  A nice example of degree distribution is   $X_n = \textrm{Bin}(n, \gamma/n)$ and $X=\textrm{Poi}(\gamma)$ for some $\gamma >0$. This case is of particular interest due the closeness between the  configuration models and  Galton-Watson trees.} \end{rem}
\section{Appendix}
\subsection{Complement of the proof of Lemma \ref{lgb}} We first recall the formula of $f(t)$
\begin{equation} \label{ftbc}
f(t)= \frac{c(1-2t)+ \sqrt{1+(c^2-1)(1-2t)^2}}{2(1-t)},
\end{equation}
which satisfies the fixed point equation
\begin{equation} \label{theta}
\theta= \frac{c(1-2t) }{1-t}  + \frac{t}{\theta(1-t)},
\end{equation} 
with $c=e^{-2 \beta}$.

\noindent {\it Proof of} \eqref{chfm} {\it and} \eqref{fpc}.  It follows from  \eqref{ftbc}  and \eqref{theta} that for all $0 \leq t \leq 1/2$,
\begin{equation} \label{sbg}
f(t) \leq \frac{(1-2t) +1}{2(1-t)} = 1
\end{equation}
and
\begin{eqnarray*}
f'(t)= \frac{-c}{(1-t)^2} + \frac{1}{f(t)(1-t)^2} - \frac{f'(t)t}{f(t)^2(1-t)}. 
\end{eqnarray*}
Hence
\begin{eqnarray} \label{fplm}
f'(t)\left(1+\frac{t}{f(t)^2(1-t)} \right)= \frac{(1/f(t))-c}{(1-t)^2} >0. 
\end{eqnarray}
Thus $f(t)$ is increasing in $(0,1/2)$. Therefore  
\begin{equation} \label{ef}
c =f(0)\leq f(t) \leq 1,
\end{equation} 
which implies \eqref{chfm}. 
It follows from  \eqref{fplm} and \eqref{ef} that for all $ t \in (0,1/2)$, 
\begin{equation} \label{efp}
\frac{1-c}{1 +1/c^2}\leq f'(t) \leq \frac{4(1-c^2)}{c}.
\end{equation}
Similarly,
\begin{eqnarray*}
f''(t)= \frac{-2c}{(1-t)^3} + \frac{2}{f(t)(1-t)^3} - \frac{2f'(t)}{f(t)^2(1-t)^2} - \frac{t}{1-t} \left( \frac{f''(t)f(t)^2-2f(t)f'(t)^2}{f(t)^4}\right). 
\end{eqnarray*}
Hence
\begin{eqnarray*}
f''(t) \left(1+ \frac{t}{(1-t)f(t)^2} \right) = \frac{2(1/f(t) -c)}{(1-t)^3} - \frac{2 f'(t)}{(1-t)f(t)^2} \left(\frac{1}{1-t}- \frac{tf'(t)}{f(t)} \right).
\end{eqnarray*}
Using this together with  \eqref{ef} and \eqref{efp}, we can show that there is a positive constant $A=A(c)$, such that for all $t \in (0,1/2)$,
\begin{equation} \label{avma}
1/A \leq f'(t) \leq A \qquad \textrm{ and }  \qquad |f''(t)| \leq A. 
\end{equation}
Thus \eqref{fpc} holds. \hfill $\square$

\vspace{0.2 cm}
\noindent {\it Proof of \eqref{fhf}}.  Let  us recall  the sequence $h(k,m)$ defined in Section 2: $h(1,m)=c$ and for $k \leq [m/2]$,
 \begin{equation} \label{nhkn}
h(k+1,m)= \frac{c(m-2k) }{m-k}  + \frac{k}{(m-k) h( k,m)}.
\end{equation}
We define 
\[K=\left[\frac{A^2c^2+2}{c^3}\right],\]
with $A$  as in \eqref{avma}.   We first claim that  for  $m \geq 4K$  and $1 \leq k \leq k_*:= [m/2] -K$,
\begin{equation} \label{fhfh}
 f((k-1)/m) \leq h(k,m) \leq f(k/m).
\end{equation}
Assuming \eqref{fhfh}, we now prove \eqref{fhf}. Let us define for $k \leq [m/2]$,
\[a_k= |h(k,m)-f((k-1)/m)|.\]
By \eqref{fhfh},   for $0 \leq k \leq k_*$ we have
\begin{equation} \label{hohf}
a_k \leq |f(k/m)-f((k-1)/m)| \leq A/m, 
\end{equation}
by using the mean value theorem and \eqref{avma}. To estimate $(a_k)$ with $k\geq k_*$, we need some bounds on $h(k,m)$. By \eqref{nhkn}, we have for all $k_* \leq k \leq [m/2]=k_*+K$,
\begin{equation*}
\frac{1}{2h(k,m)} \leq h(k+1,m) \leq c+ \frac{1}{h(k,m)}.
\end{equation*}
Moreover, $c\leq h(k_*,m) \leq 1$ by  \eqref{ef} and \eqref{fhfh}. Thus there exists a positive constant $\Theta = \Theta(K,c) \geq 1$, such that for all $k_* \leq k \leq [m/2]$
\begin{equation} \label{tahta}
1/ \Theta \leq  h(k,m) \leq \Theta.
\end{equation}
By \eqref{ftbc}, we have for $k \leq [m/2]$, 
\begin{equation} \label{fckm}
f(k/m)= \frac{c(m-2k)}{m-k} + \frac{k}{(m-k)f(k/m)}.
\end{equation}
Then using \eqref{nhkn} and \eqref{fckm}, we get  that for $k_* \leq k \leq [m/2]$,
\begin{eqnarray} \label{roa}
 a_{k+1}  &=& \frac{k}{m-k} \Big| \frac{1}{h(k,m)} - \frac{1}{f(k/m)}\Big| \notag \\
  (\textrm{use }  \eqref{ef},  \eqref{tahta}) \hspace{2 cm} & \leq & \frac{\Theta |h(k,m)-f(k/m)|}{c}  \notag \\
& \leq & \frac{ \Theta |h(k,m)-f((k-1)/m)|}{c} + \frac{ \Theta|f(k,m)-f((k-1)/m)|}{c} \notag \\
(\textrm{use }   \eqref{avma})  \hspace{3.2 cm} & \leq & \frac{\Theta a_k}{c} + \frac{\Theta A}{mc}.    
\end{eqnarray}
Using \eqref{roa}, we can prove by induction on $t$ that for all $k_* \leq k_*+t \leq  [m/2]$, 
\begin{equation} \label{akmk}
a_{k_*+t} \leq  \left(\frac{\Theta}{c}\right)^t a_{k_*} + \frac{\Theta A}{mc} \sum \limits_{i=0}^{t-1}\left( \frac{\Theta }{c}\right)^i \leq 
\left(\frac{\Theta }{c}\right)^t a_{k_*} + \frac{A}{m}\left( \frac{\Theta }{c}\right)^{t+1}.
\end{equation}
Using \eqref{hohf} and \eqref{akmk}, we obtain that for all $k\leq [m/2]$,
\[a_k \leq \varkappa/m,\]
with 
$$\varkappa = A\left[\left(\frac{\Theta }{c}\right)^K  + \left( \frac{\Theta }{c}\right)^{K+1}\right],$$
which implies \eqref{fhf}.

We now prove  \eqref{fhfh} by induction on $k$.  For $k=1$, we have 
\[c=h(1,m)=f(0/m)  \leq f(1/m), \]
since $f(t)$ is increasing. Suppose that \eqref{fhfh} holds for all $ k \leq k_*-1= [m/2]-K-1$.  We now show that it holds  for $k+1$. Using \eqref{nhkn} and \eqref{fckm} and  $h(k,m) \leq f(k/m)$, we get
\begin{eqnarray} \label{hgf}
h(k+1,m) &=& \frac{c(m-2k)}{m-k} + \frac{k}{(m-k)h(k,m)} \notag\\
&\geq & \frac{c(m-2k)}{m-k} + \frac{k}{(m-k)f(k/m)} \notag\\
&=& f(k/m).
\end{eqnarray} 
Similarly, using   $f((k-1)/m) \leq h(k,m)$, we obtain
\begin{eqnarray*} 
h(k+1,m) &\leq & \frac{c(m-2k)}{m-k} + \frac{k}{(m-k)f((k-1)/m)} \notag \\
&=& f(k/m) + \frac{k}{m-k} \left(\frac{1}{f((k-1)/m)} -\frac{1}{f(k/m)}\right) \\
&\leq& f(k/m) + \frac{k}{m-k} \left(\frac{f(k/m) -f((k-1)/m)}{f((k-1)/m)^2} \right),
\end{eqnarray*}
since $f(t)$ is increasing in $[0,1/2]$. Let us define for $k \leq k_*$,
$$b_k=f(k/m) -f((k-1)/m).$$
Then
\begin{eqnarray} \label{hlf}
\hspace{-0.4 cm} f((k+1)/m) -h(k+1,m)  &\geq& b_{k+1} -  \frac{k}{m-k} \frac{b_k}{f((k-1)/m)^2}  \notag \\
&=& b_{k+1} - b_k + b_k \left(1-  \frac{k}{(m-k)f((k-1)/m)^2} \right).
\end{eqnarray}
Using the mean value theorem, we get
\begin{eqnarray} \label{bk}
b_k= \frac{f'(y_k)}{m} \hspace{2 cm} \textrm{and} \hspace{2 cm} b_{k+1}= \frac{f'(y_{k+1})}{m},
\end{eqnarray} 
for some $y_k \in ((k-1)/m,k/m)$ and $y_{k+1} \in (k/m,(k+1)/m)$.  Using  the mean value theorem,   \eqref{avma} and the fact that $|y_k-x_k| \leq 2/m$, we have
\begin{eqnarray} \label{bkk}
b_{k+1}-b_k=\frac{f'(y_{k+1})- f'(y_k)}{m} \geq \frac{-2 }{m^2} \max_{y_k \leq t \leq y_{k+1}}|f''(t)| \geq  \frac{-2A}{m^2}.
\end{eqnarray}
Using \eqref{fckm}, we obtain
\begin{eqnarray*}
1-\frac{(k-1)}{(m-k+1)f((k-1)/m)^2} =  \frac{c(m-2k+2)}{(m-k+1)f((k-1)/m)} \geq \frac{c(m-2k+2)}{(m-k+1)}, 
\end{eqnarray*}
since $f(t) \leq 1$ for all $t\leq 1/2$. On the other hand, for $k \leq  [m/2]$
\begin{eqnarray*}
\Big|\frac{k}{(m-k)f((k-1)/m)^2} - \frac{(k-1)}{(m-k+1)f((k-1)/m)^2} \Big| \leq  \frac{4}{m f((k-1)/m)^2} \leq \frac{4}{mc^2}, 
\end{eqnarray*}
since $c \leq f(t)$ for all $t\leq 1/2$. Combining the last two inequalities yields that
\begin{eqnarray} \label{mtkf}
1-\frac{k}{(m-k) f((k-1)/m)^2} \geq  \frac{c(m-2k+2)}{(m-k+1)} -\frac{4}{mc^2} \geq  \frac{c(m-2k+2)}{m} -\frac{4}{mc^2}. 
\end{eqnarray}
It follows from  \eqref{hlf}, \eqref{bk}, \eqref{bkk}, \eqref{mtkf} that  
\begin{eqnarray} \label{fghk}
f((k+1)/m) -h(k+1,m) \geq -\frac{2A}{m^2} + \frac{f'(y_k)}{m}  \left( \frac{c(m-2k+2)}{m} -\frac{4}{mc^2} \right) \notag \\
\geq - \frac{2A}{m^2} +  \frac{1}{Am} \left( \frac{c(m-2k+2)}{m} -\frac{4}{mc^2} \right)  \geq 0,
\end{eqnarray}
by  using \eqref{avma} and the fact that
$$k \leq  [m/2]- \left[\frac{A^2c^2+2}{c^3}\right].$$
It follows from \eqref{hgf} and \eqref{fghk} that the induction step from $k$ to $k+1$ holds. Thus the proof of \eqref{fhfh} is completed.
 \hfill $\square$
\subsection{Proof of Lemma \ref{lofl}} Assume that  $t \leq 1/2$.  Using  integration by parts, we have 
\begin{eqnarray}
F(t) &=& \int_0^t \log f(s) ds = t \log f(t) - \int_0^t \frac{f'(s)}{f(s)} sds. \label{Fft}
\end{eqnarray}
We have  $f(s)= A(s)/B(s)$, where
\[A(s)= e^{-2 \beta} (1-2 s) + \sqrt{1+ (e^{-4 \beta}-1)(2s-1)^2} \qquad \textrm{and} \qquad B(s)=2(1-s).\]
Moreover,
\begin{eqnarray*}
\frac{A'(s)}{A(s)}= \frac{1}{2s(1-s)} \left[ 1-2s - \frac{e^{-2 \beta}}{\sqrt{1+(e^{-4 \beta} -1)(2s-1)^2}}  \right], 
\end{eqnarray*}
and 
\begin{eqnarray*}
\frac{B'(s)}{B(s)}= \frac{-1}{1-s}. 
\end{eqnarray*}
Hence
\begin{eqnarray} \label{fpfs}
\frac{f'(s)s}{f(s)}&= &s \left[ \frac{A'(s)}{A(s)} - \frac{B'(s)}{B(s)}\right]  \notag \\
&=&\frac{1}{2(1-s)}  - \frac{e^{-2 \beta}}{2(1-s) \sqrt{1+(e^{-4 \beta} -1)(2s-1)^2}}. 
\end{eqnarray}
Combining \eqref{Fft} and \eqref{fpfs} gives that 
\begin{eqnarray} \label{Ftt}
F(t)= t \log f(t) + \frac{1}{2} \log (1-t) + \int_0^t  \frac{e^{-2 \beta}}{2(1-s) \sqrt{1+(e^{-4 \beta} -1)(2s-1)^2}} ds.
\end{eqnarray}
Let $\alpha= \sqrt{1-e^{-4 \beta}} \in (0,1)$. Then by computation and changing variables, we have 
\begin{eqnarray*}
J&=& \int_0^t  \frac{e^{-2 \beta}}{(1-s) \sqrt{1+(e^{-4 \beta} -1)(2s-1)^2}} ds \\
(u= 1 - 2s) \hspace{0.4 cm} &=& \int_{1-2t}^1  \frac{\sqrt{1- \alpha^2}  }{(1+u) \sqrt{1- \alpha^2 u^2 }} du \\
(v= \arcsin(\alpha u)) \hspace{0.4 cm} &=& \int^{\arcsin(\alpha)}_{\arcsin(\alpha(1-2t))}  \frac{\sqrt{1- \alpha^2}   }{\alpha+ \sin v } dv  \\  
(w= \tan(v/2)) \hspace{0.4 cm} &=& \int_{w_1}^{w_2}  \frac{2\sqrt{1- \alpha^2}   }{\alpha w^2+ 2 w + \alpha }  dw \\  
&=& \log  \left( \frac{\alpha w + 1 - \sqrt{1- \alpha^2}  }{\alpha w + 1 + \sqrt{1- \alpha^2}} \right) \Big|_{w_1}^{w_2}.
\end{eqnarray*}
For $x\in (0, 1)$, we have  
$$\tan \left(\frac{\arcsin(x)}{2} \right) = \frac{1- \sqrt{1- x^2}}{x}.$$
Thus 
$$w_2= \frac{1- \sqrt{1- \alpha^2}}{\alpha} \hspace{2 cm} \textrm{and} \hspace{2 cm} w_1= \frac{1- \sqrt{1- \alpha^2(1- 2 t)^2}}{\alpha(1- 2 t)}.$$
Therefore
\begin{eqnarray*}
J&=& \log \left(1- e^{-2 \beta} \right) - \log  \left( \frac{ (1- e^{-4 \beta})(1-2t)  + ( 1- e^{-2 \beta}) [1 + \sqrt{1 +( e^{-4 \beta} -1)(1-2t)^2}]}  {(1- e^{-4 \beta})(1-2t)  + ( 1+ e^{-2 \beta}) [1 + \sqrt{1 +( e^{-4 \beta} -1)(1-2t)^2}]} \right) \notag \\
&=& \log \left(1+ e^{-2 \beta} \right) - \log  \left( \frac{ (1+ e^{-2 \beta})(1-2t)  + 1 + \sqrt{1 +( e^{-4 \beta} -1)(1-2t)^2}}  {(1- e^{-2 \beta})(1-2t)  + 1 + \sqrt{1 +( e^{-4 \beta} -1)(1-2t)^2}} \right) \notag \\
&=& \log \left(1+ e^{-2 \beta} \right) + \log \left( \frac{1- t + t f(1-t)}{(1-t)(f(t)+1)} \right).
\end{eqnarray*}
Combining this with \eqref{Ftt}, we obtain 
\begin{eqnarray} \label{fltb} 
F(t)= t \log f(t) + \frac{1}{2} \log (1-t) + \frac{1}{2}\log (1+e^{-2 \beta})   + \frac{1}{2}\log \left[ 1+ \frac{e^{-2 \beta}(2t-1)}{(1-t)(f(t)+1)} \right].
\end{eqnarray}
Thus Lemma \ref{lofl} follows. \hfill $\square$
\subsection{Proof of Proposition \ref{eopp}} We first recall the formula for  the quenched pressure determined in \cite{DM}.
Suppose that $ \beta>0 $ and $B>0$. Let $h_*$ be the positive solution of a fixed point equation:
\begin{equation} \label{h}
h=B+ (d-1) \textrm{atanh} (\tanh(\beta) \tanh(h)).
\end{equation}  
Then 
\begin{eqnarray} \label{etps}
\tilde{\psi}(\beta, B) &= &\frac{d}{2} \log (\cosh(\beta)) - \frac{d}{2} \log \left(1+ \tanh(\beta) \tanh(h_*)^2 \right) \\ 
&& + \log \left[ e^B(1+ \tanh(\beta) \tanh(h_*))^d + e^{-B}(1- \tanh(\beta) \tanh(h_*))^d \right]. \notag
\end{eqnarray}
For $B<0$, one has $\tilde{\psi}(\beta, B) = \tilde{\psi}(\beta, -B)$, and   $\tilde{\psi}(\beta, 0) = \lim \limits_{B \downarrow 0} \tilde{\psi}(\beta, B)$. 
In this subsection, we will show that
\begin{equation} \label{mbmt}
 \tilde{\psi}(\beta, B)=\psi(\beta, B).
\end{equation}
We prove  here the case $B>0$, then  the other case follows from the fact that both of the functions $ \tilde{\psi}(\beta, \cdot)$ and $\psi(\beta, \cdot)$ are  even. We have proved in Sections 3 and 4 that for   $B >0$, 
\begin{equation} \label{ezng}
\psi(\beta,B)= \frac{\beta d}{2} -B +   L(t_*),
\end{equation}
where  
\[L(t)= -t \log (t) + (t-1) \log (1-t) + 2 Bt + d F(t), \]
and $t_* \in (1/2, 1)$ is the unique solution of  the equation 
\begin{equation} \label{lptk}
L'(t)= \log\left(\frac{1-t}{t}\right) - d  \log f(1-t) +2B=0.
\end{equation}
We claim a relation  between $h_*$ and $t_*$, which will prove later.   
\begin{equation} \label{etvs}
 2t_* - 1= \tanh (h_*+ \atanh (\tanh (\beta) \tanh(h_*))).
\end{equation}
Assuming \eqref{etvs}, we now prove  \eqref{mbmt}.

\noindent {\bf Expression of $\tilde{\psi}(\beta, B)$}. Let us denote 
$$u_* = \tanh(\beta) \tanh(h_*).$$
Applying the  function tanh to both sides of the equation \eqref{h}, we get 
 \begin{eqnarray} \label{tohs}
 \tanh(h_*) &=& \tanh (B + (d-1) \atanh (u_*)) \notag\\
&=& \frac{e^{2B}(1+u_*)^{d-1}-(1-u_*)^{d-1}}{e^{2B}(1+u_*)^{d-1}+(1-u_*)^{d-1}}.
 \end{eqnarray}
 Thus 
 \begin{eqnarray*}
 1+ \tanh(\beta) \tanh(h_*)^2= 1+ u_* \tanh(h_*) = \frac{e^{2B}(1+u_*)^{d}+(1-u_*)^{d}}{e^{2B}(1+u_*)^{d-1}+(1-u_*)^{d-1}}.
 \end{eqnarray*}
 Therefore, using \eqref{etps} we have 
 \begin{eqnarray} \label{tpsib}
 \tilde{\psi}(\beta, B) &=& \frac{d}{2} \log (\cosh(\beta)) - \frac{d}{2} \log \left[\frac{e^{2B}(1+u_*)^{d}+(1-u_*)^{d}}{e^{2B}(1+u_*)^{d-1}+(1-u_*)^{d-1}} \right] \notag \\
 && + \log  \left[e^{B}(1+u_*)^{d}+e^{-B}(1-u_*)^{d} \right].
 \end{eqnarray}
 {\bf Expression of $\psi(\beta, B)$}. We first display $t_*$ and $e^{-2 \beta}$ in term of $u_*$. Using \eqref{etvs}, we have 
 \begin{eqnarray} \label{htstm}
 2t_* - 1= \tanh (B+ d \atanh (u_*)) =  \frac{e^{2B}(1+u_*)^{d}-(1-u_*)^{d}}{e^{2B}(1+u_*)^{d}+(1-u_*)^{d}}. 
 \end{eqnarray}
Thus 
\begin{eqnarray} \label{tvts}
t_* = \frac{e^{2B}(1+u_*)^{d}}{e^{2B}(1+u_*)^{d}+(1-u_*)^{d}} \quad \textrm{ and } \quad 1 -t_* =  \frac{(1-u_*)^{d}}{e^{2B}(1+u_*)^{d}+(1-u_*)^{d}}.
\end{eqnarray}
 On the other hand,  using \eqref{tohs} we get
 \begin{eqnarray} \label{bus}
 e^{-2 \beta} = \frac{1-\tanh(\beta)}{1+\tanh(\beta)} = \frac{1- \frac{u_*}{ \tanh(h_*)}}{1+ \frac{u_*}{ \tanh(h_*)}} = (1-u_*^2)\frac{e^{2B}(1+u_*)^{d-2}-(1-u_*)^{d-2}}{e^{2B}(1+u_*)^{d}-(1-u_*)^{d}}.
 \end{eqnarray}
Since $t_* >1/2$, we have $F(t_*)=F(1-t_*)$. Thus using  \eqref{fltb},  
 \begin{eqnarray} \label{fmtt} 
F(t_*)= F(1-t_*) &=& (1-t_*) \log f(1-t_*) + \frac{1}{2} \log t_* + \frac{1}{2}\log (1+e^{-2 \beta})  \notag \\ 
 && + \frac{1}{2}\log \left[ 1+ \frac{e^{-2 \beta}(1-2t_*)}{t_*(f(1-t_*)+1)} \right].
\end{eqnarray}
Since $t_*$ is the solution of \eqref{lptk}, we have
\begin{eqnarray} \label{fmtsb}
f(1-t_*) = e^{\frac{2B}{d}} \left(\frac{1-t_*}{t_*} \right)^{\frac{1}{d}}.
\end{eqnarray}
Using \eqref{fmtsb} and  \eqref{tvts}, 
\begin{eqnarray} \label{fmtu}
f(1-t_*) = \frac{1-u_*}{1+u_*}. 
\end{eqnarray}
Combining  \eqref{htstm}, \eqref{tvts}, \eqref{bus} and \eqref{fmtu} yields that 
\begin{equation} \label{caid}
1+ \frac{e^{-2 \beta}(1-2t_*)}{t_*(f(1-t_*)+1)} = \frac{(1+u_*)^2}{2} \times \frac{ e^{2B}(1+u_*)^{d-1}+(1-u_*)^{d-1}}{e^{2B}(1+u_*)^d}.
\end{equation}
Using \eqref{fmtt}, \eqref{fmtsb} and \eqref{caid},
\begin{eqnarray*}
dF(t_*) &=& (1-t_*) \left( 2B + \log \left( \frac{1-t_*}{t_*}\right) \right) +  \frac{d}{2}\log (1+e^{-2 \beta}) \notag \\
&&+ d \log (1+u_*) - \frac{d}{2} \log 2 + \frac{d}{2} \log \left(t_* \frac{e^{2B}(1+u_*)^{d-1}+(1-u_*)^{d-1}}{e^{2B}(1+u_*)^d}\right)  \notag \\
&=& (2-2t_*)B  + (1-t_*) \log \left( \frac{1-t_*}{t_*}\right)  +  \frac{d}{2}\log \big[(1+e^{-2 \beta})/2\big] \notag \\
(\textrm{use } \eqref{tvts}) \hspace{1 cm}  &&+ d \log (1+u_*)  + \frac{d}{2} \log \left( \frac{e^{2B}(1+u_*)^{d-1}+(1-u_*)^{d-1}}{e^{2B}(1+u_*)^d+(1-u_*)^{d}}\right).  
\end{eqnarray*}
Hence 
\begin{eqnarray*}
&&\psi(\beta, B) = \frac{\beta d}{2} - B + (t_*-1) \log (1-t_*) -t_* \log t_* + 2B t_* +d F(t_*) \\
&=&\frac{d}{2} \left( \beta + \log \big[(1+e^{-2 \beta})/2\big] \right) +B - \log t_* + d \log (1+u_*)  + \frac{d}{2} \log \left( \frac{e^{2B}(1+u_*)^{d-1}+(1-u_*)^{d-1}}{e^{2B}(1+u_*)^d+(1-u_*)^{d}}\right) \notag \\
&=&\frac{d}{2} \log( \cosh(\beta))+ \log \left(e^{B} t_*^{-1}(1+u_*)^d \right) + \frac{d}{2} \log \left( \frac{e^{2B}(1+u_*)^{d-1}+(1-u_*)^{d-1}}{e^{2B}(1+u_*)^d+(1-u_*)^{d}}\right) \notag \\
&=&\frac{d}{2} \log( \cosh(\beta))+  \log \left( e^{B}(1+u_*)^d+ e^{-B}(1-u_*)^{d}\right)  + \frac{d}{2} \log \left( \frac{e^{2B}(1+u_*)^{d-1}+(1-u_*)^{d-1}}{e^{2B}(1+u_*)^d+(1-u_*)^{d}}\right),
\end{eqnarray*}
where for the last line, we used \eqref{tvts}. Using this equation with \eqref{tpsib}, we obtain that $$\psi(\beta, B)= \tilde{\psi}(\beta, B),$$
which proves \eqref{mbmt}. For $d=2$, it has been shown in \cite{GGHPa, GGHPb} that 
\[\psi(\beta, B) = \tilde{\psi}(\beta, B)= \beta + \log \left( \cosh (B) + \sqrt{ \sinh ^2 (B) + e^{-4 \beta}} \right). \]

\noindent {\it Proof of \eqref{etvs}}. Let us denote
$$v_*= \tanh (h_*+ \atanh (\tanh (\beta) \tanh(h_*))).$$
We claim  the following identity {\bf (E)}:  For all $x>0$ and $y \in \R$, if 
\begin{equation} \label{vty}
v= \tanh (y+ \atanh (\tanh (x) \tanh(y))), 
\end{equation}
then
\begin{equation} \label{ecsh}
 \frac{e^{-2 x  }v+ \sqrt{1+(e^{-4x} -1)v^2}}{v+1} = \frac{\cosh(x-y)}{\cosh(x+y)}.
\end{equation} 

\vspace{0.2 cm}
\noindent Assuming {\bf (E)}, we can prove \eqref{etvs}. Indeed, using  \eqref{ecsh} we get 
\begin{equation} \label{csvs}
 \frac{e^{-2 \beta  }v_*+ \sqrt{1+(e^{-4\beta} -1)v_*^2}}{v_*+1} = \frac{\cosh(\beta-h_*)}{\cosh(\beta+h_*)}.
\end{equation} 
Since $h_*$ is the solution of \eqref{h}, we have
\begin{equation*} \label{vsa}
v_*= \tanh (B+  d \atanh (\tanh (\beta) \tanh(h_*))).
\end{equation*}
Applying the function atanh to the both sides of the above equation, we obtain
\begin{eqnarray*}
\frac{1}{2} \log \left( \frac{1+v_*}{1-v_*} \right)= B + \frac{d}{2}   \log \left(  \frac{\cosh(\beta+h_*)}{\cosh(\beta-h_*)}\right).
\end{eqnarray*}
Combining this with \eqref{csvs}, we have 
\begin{eqnarray} 
 \log \left( \frac{1-v_*}{1+v_*} \right)  - d \log \left( \frac{e^{-2 \beta  }v_*+ \sqrt{1+(e^{-4\beta} -1)v_*^2}}{v_*+1}\right)+2B = 0,
\end{eqnarray}
or equivalently, by using \eqref{lptk}
\begin{equation} \label{eovs}
L' \left( \frac{v_*+1}{2} \right)=0.
\end{equation}
It follows from \eqref{lptk} and \eqref{eovs} that $t_*$ and $(v_*+1)/2$ are  solutions in $ (1/2,1) $ of the  equation $L'(x)=0$. We have proved in Claim $1^*$ in Section 4 that this equation has unique solution. Thus  $t_*=(v_*+1)/2$, and \eqref{etvs} follows. 

\vspace{0.2 cm}
\noindent We now prove the identity {\bf (E)}.  Applying the function atanh to the both sides of  \eqref{vty} gives that  
\begin{eqnarray} \label{mvv}
\frac{1+v}{1-v} =  \frac{ e^{2 y }  \cosh(x+y)}{\cosh(x-y)}.
\end{eqnarray}
Hence, \eqref{ecsh} is equivalent to 
\begin{equation*}
(1-v)e^{2y} = e^{-2 x} v +  \sqrt{1+(e^{-4x} -1)v^2},
\end{equation*}
or 
\begin{equation} \label{cav}
\begin{cases} 
(1-v)e^{2y} - e^{-2 x} v \geq 0 \\
\left( (1-v)e^{2y} - e^{-2 x} v \right)^2 = 1+(e^{-4x} -1)v^2
\end{cases}
\end{equation}
We have 
\begin{eqnarray*}
(1-v)e^{2y} - e^{-2 x} v \geq 0 &\Leftrightarrow&  e^{2x+2y} \geq \frac{v}{1-v}. 
\end{eqnarray*}
On the other hand 
\begin{eqnarray*}
\frac{v}{1-v} \leq \frac{1+v}{1-v} = \frac{ e^{2 y }  \cosh(x+y)}{\cosh(x-y)} = \frac{e^{y-x}+e^{3y+x}}{e^{x-y}+e^{y-x}} \leq  e^{2x+2y},
\end{eqnarray*}
since $x>0$. Hence, the  inequality in \eqref{cav} holds.  The  equation in \eqref{cav} is equivalent to
\begin{eqnarray*}
&&(1-v)^2e^{4y} - 2v(1-v) e^{2y-2 x} = 1-v^2 \\
&\Leftrightarrow & \left( \frac{1-v}{1+v} \right) e^{4y} - \frac{2v}{1-v} e^{2y-2x} = 1 \\
&\Leftrightarrow & \left( \frac{1-v}{1+v} \right) e^{4y} + \left(\frac{1-v}{1+v}-1 \right) e^{2y-2x} = 1 \\
&\Leftrightarrow & \left( \frac{1-v}{1+v} \right) (e^{4y} + e^{2y-2x})   = 1 + e^{2y-2x}. \\
\textrm{(using \eqref{mvv})} \hspace{0.6 cm} &\Leftrightarrow & e^{-2y} \left( \frac{e^{x-y}+e^{y-x}}{e^{x+y}+e^{-x-y}} \right) (e^{4y} + e^{2y-2x})   = 1 + e^{2y-2x}. \\
&\Leftrightarrow & e^{3y-x} + e^{y-3x} = e^{3y-x} + e^{y-3x},
\end{eqnarray*}
which holds for all $x,y$. In conclusion, \eqref{cav} holds, and thus \eqref{ecsh} follows.  \hfill $\square$

\begin{ack} \emph{
 I would like to thank Bruno Schapira, Arnaud Le Ny and Trung Ha  for their  helpful discussions. I wish to thank also anonymous referees for carefully reading the manuscript and many valuable comments.}
 \end{ack}

\end{document}